\documentclass[english]{amsart}

\usepackage{amsmath}

\usepackage{url}

\usepackage{amssymb}

\usepackage{amscd} \usepackage{tabularx}
\usepackage[all,cmtip,line]{xy}

\newcommand{\F}{\mathbb{F}} 
\def\RCS$#1: #2 ${\expandafter\def\csname RCS#1\endcsname{#2}}
\RCS$Revision: 166 $ \RCS$Date: 2009-05-25 15:31:15 -0400 (Mon, 25 May
2009) $ 
\DeclareMathOperator{\Frob}{Frob}

 \DeclareMathOperator{\ab}{ab}
 
\newcommand{\To}{\longrightarrow}\newcommand{\into}{\hookrightarrow}
\newcommand{\m}{\mathfrak{m}}
\newcommand{\psibar}{\overline{\psi}}
\newcommand{\onto}{\twoheadrightarrow}
\newcommand{\isoto}{\stackrel{\sim}{\To}} 
 
 \newcommand{\rec}{\operatorname{rec}}
 
\newcommand{\bigO}{\mathcal{O}} 
\newcommand{\G}{\mathcal{G}}

\newcommand{\Z}{\mathbb{Z}} \newcommand{\A}{\mathbb{A}}
\newcommand{\Q}{\mathbb{Q}}

\newcommand{\C}{\mathbb{C}} 
\newcommand{\Gn}{\mathcal{G}_n}
\newcommand{\sq}{\square}

\newcommand{\bb}{\mathbb} 
\newcommand{\mc}{\mathcal}
\newcommand{\wt}{\widetilde} 
\newcommand{\mf}{\mathfrak}
\newcommand{\Uop}{U_{\lambda,\varpi_{\wt{v}}}^{(j)}}
\newcommand{\TUniv}{\bb{T}^{T,\ord}(U(\mf{l}^{\infty}),\mc{O})}
\DeclareMathOperator{\Iw}{Iw}
\DeclareMathOperator{\ord}{ord}  
\DeclareMathOperator{\Art}{Art}
\DeclareMathOperator{\univ}{univ}
\DeclareMathOperator{\red}{red}

\DeclareMathOperator{\sesi}{ss}
\DeclareMathOperator{\id}{id}

\newcommand{\rhobar}{\overline{\rho}} \newcommand{\rbar}{\bar{r}}
\newcommand{\rbarwtv}{\rbar|_{G_{F_{\wt{v}}}}}
\newcommand{\rmbarwtv}{\rbar_{\mf{m}}|_{G_{F_{\wt{v}}}}}
\newcommand{\mubar}{\overline{\mu}}

\newcommand{\Gal}{\operatorname{Gal}}
\newcommand{\GL}{\operatorname{GL}}
\newcommand{\GSp}{\operatorname{GSp}}
\newcommand{\Sp}{\operatorname{Sp}}

 \newcommand{\Qbar}{\overline{\Q}}
 \newcommand{\Qp}{\Q_p}
\newcommand{\Ql}{\Q_l} 

\newcommand{\Qlbar}{\overline{\Q}_{l}}

\newcommand{\Flbar}{\overline{\F}_l} 

\newcommand{\tv}{{\tilde{v}}}

\newcommand{\gr}{\operatorname{gr}}

\newcommand{\Spec}{\operatorname{Spec}}

\newcommand{\ad}{\operatorname{ad}}

\newcommand{\tr}{\operatorname{tr}}
\newcommand{\coker}{\operatorname{coker}}
\newcommand{\End}{\operatorname{End}}
\newcommand{\Hom}{\operatorname{Hom}}

\newcommand{\Fil}{\operatorname{Fil}}

\newcommand{\diag}{\operatorname{diag}}

\newcommand{\Gm}{{\mathbb{G}_m}}

\DeclareMathOperator{\Def}{Def}
\DeclareMathOperator{\cris}{cris}

\usepackage{amsthm}

\newtheorem{thm}{Theorem}[subsection]

\newtheorem{cor}[thm]{Corollary}
 \newtheorem{lemma}[thm]{Lemma}
\newtheorem{lem}[thm]{Lemma} \newtheorem{prop}[thm]{Proposition}
 \theoremstyle{definition}
 \theoremstyle{definition}
\newtheorem{defn}[thm]{Definition} \theoremstyle{remark}
\newtheorem{rem}[thm]{Remark} \numberwithin{equation}{subsection}
\newtheorem{remark}[thm]{Remark} \numberwithin{equation}{subsection}

\theoremstyle{definition}\newtheorem*{notation}{Notation}

\begin{document}
\title[Companion forms for unitary groups]{Companion forms for unitary
  and symplectic groups}

\author{Toby Gee} \email{tgee@math.harvard.edu} \address{Department of
  Mathematics, Harvard University} \author{David Geraghty}
\email{geraghty@math.harvard.edu}\address{Department of Mathematics,
  Harvard University} \thanks{The first author was partially supported
  by NSF grant DMS-0841491.}  
\begin{abstract}We prove a companion forms theorem for ordinary
  $n$-dimensional automorphic Galois representations, by use of
  automorphy lifting theorems developed by the second author, and a
  technique for deducing companion forms theorems due to the first author. We deduce
  results about the possible Serre weights of mod $l$ Galois
  representations corresponding to automorphic representations on
  unitary groups. We then use functoriality to prove similar results for
  automorphic representations of $\GSp_4$ over totally real fields.
\end{abstract}
\maketitle
\tableofcontents
\section{Introduction.}

\subsection{}
The problem of companion forms was first introduced by Serre for
modular forms in his seminal paper \cite{MR885783}. Fix a prime $l$,
algebraic closures $\Qbar$ and $\Qlbar$ of $\Q$ and $\Ql$ respectively,
and an embedding of $\Qbar$ into $\Qlbar$. Suppose that $f$ is a
modular newform of weight $k\ge 2$ which is ordinary at $l$, so that
the corresponding $l$-adic Galois representation $\rho_{f,l}$ becomes
reducible when restricted to a decomposition group $G_{\Ql}$ at $l$. Then the
companion forms problem is essentially the question of determining for
which other weights $k'$ there is an ordinary newform $g$ of weight
$k'\ge 2$ such that the Galois representations $\rho_{f,l}$ and
$\rho_{g,l}$ are congruent modulo $l$. The problem is straightforward
unless the restriction to $G_{\Ql}$ of $\rhobar_{f,l}$ (the reduction mod $l$ of $\rho_{f,l}$) is
split and non-scalar, in which case there are two possible Hida families whose
corresponding Galois representations lift $\rhobar_{f,l}$; the
restrictions of the corresponding Galois representations to a
decomposition group at $l$ are either ``upper-triangular'' or
``lower-triangular''.

This problem was essentially resolved by Gross and Coleman-Voloch
(\cite{MR1074305}, \cite{MR1185584}). In the paper \cite{gee051}, the
first author reproved these results, and generalised them to Hilbert
modular forms, by a completely new technique. In essence, rather than
working directly with modular forms, the method is to firstly obtain a
Galois representation which should correspond to a modular
form in the sought-after Hida family, and then to use a modularity lifting theorem to prove that this
Galois representation is modular. In \cite{gee051} the Galois
representation is obtained by using a generalisation of a lifting
technique of Ramakrishna, which is proved by purely deformation theory
techniques. The modularity is then obtained from the $R=T$ theorem of
Kisin for Hilbert modular forms of parallel weight 2 (\cite{kis04}).

These techniques seem amenable to generalisation (to other reductive
groups over more general number fields), subject to some important caveats. In particular, it
is necessary to have modularity lifting theorems available over fields in which $l$
is highly ramified. The current technology for modularity lifting theorems requires
one to work with reductive groups which admit discrete series, and to
work over totally real or CM fields; so it is impossible at present to
work directly with $\GL_n$ for $n>2$. Instead, one works with closely
related groups, such as unitary or symplectic groups, which do admit
discrete series. 

In the present paper we make use of $R=T$ theorems for unitary groups
to deduce companion forms theorems for unitary groups (in arbitrary
dimension), and thus for conjugate self-dual automorphic
representations of $\GL_n$ over CM fields. We then deduce similar
theorems for $\GSp_4$ by developing the relevant deformation theory
and employing known instances of functoriality. The analogue for
unitary groups of the $R=T$ theorems of \cite{kis04} seem to be out of
reach at present, and we use the main theorems of \cite{ger}
instead. As explained below, this in fact allows us to prove stronger
theorems than the natural analogue of \cite{kis04} would permit. We
replace the use of Ramakrishna's techniques in \cite{gee051} with a
method of Khare and Wintenberger, which allows weaker hypotheses on
local deformation problems.

To our knowledge the only results on companion forms for groups other
than $\GL_2$ are those announced for $\GSp_4$ over $\Q$ in
\cite{herzigtilouine} (see also \cite{tilouinecompanionforms}). 
Our results are rather stronger
than those of \cite{herzigtilouine} in several respects. We are able
to work with arbitrary totally real fields (with no restriction on
ramification at $l$), rather than just over $\Q$, and we do not need
any assumption that the residual Galois representation occurs at
minimal level (indeed, one may deduce results on level lowering for
$\GSp_4$ from our theorem). In addition, the results of
\cite{herzigtilouine} apply only in one special case, effectively one
of 8 cases (corresponding to the 8 elements of the Weyl group of
$\GSp_4$) where one could hope to prove a companion forms theorem;
this is in part due to the fact that their techniques only apply
to Galois representations in the Fontaine-Laffaille range. In
contrast, we make no such restrictions.  We hope that these
results will prove useful for generalisations of the Buzzard-Taylor
method to $\GSp_4$, as part of a program of Tilouine.

In recent years there has been a good deal of interest in
generalisations of Serre's conjecture (cf. \cite{MR1896473}) and in
particular in the question of determining the set of weights of a
given Galois representation (cf. \cite{herzigthesis}). One of us
(T.G.) has formulated a conjecture to the effect that the set of
weights should be determined completely by the existence of (local)
crystalline lifts (cf. \cite{gee061}). In general this seems to be a
very difficult conjecture to prove, but our methods give a substantial
partial result; essentially we prove the conjecture (subject to mild
technical hypotheses) for ordinary weights for unitary groups which
are compact at infinity. See section \ref{sec:serre weights} for the
precise statements.

We now outline the structure of the paper. In section \ref{sec:galois
  deformations} we develop the basic deformation theory that we need. We then
recall in section \ref{sec:ordinary
  automorphic representations} the necessary material on ordinary
automorphic representations on unitary groups and modularity lifting
theorems for the corresponding Galois representations; in
particular we recall the main theorem of \cite{ger}.

Section \ref{sec:existence of lifts} contains our main theorems for
unitary groups; the corresponding Galois representations are conjugate
self-dual representations of the absolute Galois group of an imaginary CM
field. Using the results of section \ref{sec:galois deformations} we
give a lower bound for the dimension of a universal deformation ring,
and the results of section \ref{sec:ordinary
  automorphic representations} then permit us to prove that this
universal deformation ring is finite over $\Z_l$, which implies that it
has $\Qlbar$-points, which correspond to the Galois representations we
seek. The automorphy of these Galois representations follows at once
from the modularity lifting theorems recalled in section \ref{sec:ordinary
  automorphic representations}. The particular universal deformation ring we consider is one for
representations of the absolute Galois group of a totally real field,
valued in a group $\Gn$ defined in \cite{cht}. Representations valued
in this group correspond to representations which are self-dual with
respect to some pairing; this permits us to prove results for both the conjugate
self-dual representations considered in section \ref{sec:existence of
  lifts}, and the symplectic representations studied in later
sections.

We remark that the $\Qlbar$-points of universal deformation rings that
we study in section \ref{sec:existence of lifts} always correspond to
ordinary crystalline representations of a certain weight. This is in
contrast to the approach of \cite{gee051}, which used potentially
crystalline representations corresponding to Hilbert modular forms of
parallel weight 2 and non-trivial level at $l$. The required automorphic representations were then
obtained by specialising Hida families through these points at the
sought-for weight. The difficulty with following this approach in
general is that if the weight is not sufficiently regular a
specialisation of a Hida family at this weight may fail to be an
unramified principal series at places dividing $l$ (for example, a specialisation of a Hida family
of modular forms in weight 2 can correspond to a Steinberg representation 
at $l$). It is for this reason that we use modularity lifting theorems
for crystalline lifts instead. 

In section \ref{sec:serre weights} we deduce results about the
possible Serre weights of mod $l$ Galois representations corresponding
to automorphic representations of compact at infinity unitary
groups. In particular, we deduce that the possible ordinary weights are determined by the existence of
local crystalline lifts. We remark that these are the first results in
anything approaching this level of generality for any groups other
than $\GL_2$.

Finally in section \ref{sec:gsp4} we study the analogous questions for
automorphic representations of $\GSp_4$ over totally real fields. We
use the known functoriality between globally generic cuspidal
representations of $\GSp_4$ and $\GL_4$ to apply the methods of the
earlier sections. In particular, we prove results analogous to those
of section \ref{sec:galois deformations} for Galois representations
valued in $\GSp_4$, and obtain a lower bound for the dimension of a
universal deformation ring as in section \ref{sec:existence of
  lifts}. We then prove that this universal deformation ring is finite
over the corresponding one for unitary representations, which allows
us to deduce that our symplectic universal deformation ring is also
finite over $\Z_l$. Our main results for symplectic representations
follow from this.

We remark that in all our main theorems we actually obtain somewhat
more precise results; we are also able to control the ramification of
our Galois representations at places not dividing $l$, and we are able
to choose our Galois representations so as to correspond to points on
any particular set of irreducible components of the local deformation
rings. Thus as a direct corollary of our results one obtains strong
results on level lowering and level raising for ordinary automorphic
Galois representations. Similarly, our method yields modularity
lifting theorems for ordinary representations of $\GSp_4$ which are
rather stronger than those of \cite{MR2234862}; for example, we do not
need to assume any form of level-lowering for $\GSp_4$, we work over
general totally real fields, and we are not restricted to weights in
the Fontaine-Laffaille range.

We would like to thank Wee Teck Gan, Florian Herzig, Mark Kisin and Richard Taylor for helpful conversations.

\section{Notation}\label{sec:notation}If $M$ is a field, we let $G_M$ denote its absolute
Galois group. 
Let $\epsilon$ denote the $l$-adic or mod $l$
cyclotomic character of $G_M$. If $M$ is a finite extension of $\bb{Q}_p$ for
some $p$, we write $I_M$ for the inertia subgroup of $G_M$. 
We write all matrix transposes on the left; so ${}^tA$ is
the transpose of $A$. 
If $R$ is a local ring we write $\mf{m}_{R}$ for the maximal ideal of
$R$. We let $\bb{Z}^n_+$ denote the subset of elements $\lambda \in
\bb{Z}^n$ with $\lambda_1 \geq \ldots \geq \lambda_n$.

\section{Galois deformations}\label{sec:galois deformations}
\subsection{Local deformation rings}
\label{LocalDefRings}
Let $l$ be a prime number and $K$ a finite extension of
$\Ql$ with residue field $k$ and ring of integers $\bigO$. 
Let $M$ be a
finite extension of $\Qp$ (with $p$ possibly equal to $l$). 
Let
$\rhobar:G_M\to\GL_n(k)$ be a continuous representation. Let
$\mathcal{C}_\bigO$ be the category of complete local Noetherian $\bigO$-algebras
with residue field $k$. Then the functor from $\mathcal{C}_\bigO$ to
$Sets$ which takes $A\in\mathcal{C}_\bigO$ to the set of liftings of
$\rhobar$ to a continuous homomorphism $\rho:G_M\to\GL_n(A)$ is represented by
a complete local Noetherian $\bigO$-algebra $R^{\square}_{\rhobar}$. We call this
ring the universal $\mc{O}$-lifting ring of $\rhobar$.
 We write
$\rho^{\square}:G_M\to\GL_n(R^{\square}_{\rhobar})$ for the universal lifting. We will need
to consider certain quotients of $R^\square_{\rhobar}$.

\subsubsection{The case where $p \neq l$}

Firstly, we consider the case $p\neq l$. In this case, the quotients
we wish to consider will correspond to particular inertial
types. Recall that $\tau$ is an \emph{inertial type} for $G_M$ over
$K$ if $\tau$ is a $K$-representation of $I_M$ with open kernel which
extends to a representation of $G_M$, and that we say that an $l$-adic
representation of $G_M$ has type $\tau$ if the restriction of the
corresponding Weil-Deligne representation to $I_M$ is equivalent to
$\tau$. For any such $\tau$ there is a unique reduced, $l$-torsion
free quotient $R^{\square,\tau}_{\rhobar}$ of $R^{\square}_{\rhobar}$
with the property that if $E/K$ is a finite extension, then a map of
$\mc{O}$-algebras $R^{\square}_{\rhobar}\to E$ factors through
$R^{\square,\tau}_{\rhobar}$ if and only if the corresponding
$E$-representation has type $\tau$. Furthermore, we have:

\begin{lem}
\label{lem: local dimension for l not p}
  For any $\tau$, if $R^{\square,\tau}_{\rhobar}\neq 0$ then
  $R^{\square,\tau}_{\rhobar}[1/l]$ is equidimensional of dimension
  $n^2$ and is generically formally smooth.
\end{lem}
\begin{proof}
  This is Theorem 2.0.6 of \cite{gee061}.
\end{proof}

Of course, $R^{\square,\tau}_{\rhobar}\neq 0$ if and only if $\rhobar$
has a lift of type $\tau$.

\subsubsection{The case where $p=l$}
\label{subsubsec: the case where p=l}
Now assume that $p=l$. In this case, we wish to consider crystalline
ordinary deformations of fixed weight. We assume from now on that $K$ is 
large enough that any embedding $M \hookrightarrow \overline{K}$ has image
contained in $K$.

\begin{notation}
Recall that $\bb{Z}^n_+$ is the set of non-increasing $n$-tuples of integers.
 We say that $\lambda
  \in (\bb{Z}^{n}_{+})^{\Hom(M,K )}$ is \emph{regular} if for each
  $j=1,\ldots,n-1$ there exists $\tau : M \hookrightarrow K$ with
  $\lambda_{\tau,j} > \lambda_{\tau,j+1}$.

Let $\epsilon$ be the $l$-adic cyclotomic character and let $\Art_{M} :
M^{\times} \rightarrow W_{M}^{\ab}$ be the Artin map (normalized to
take uniformizers to lifts of geometric Frobenius).

\begin{defn}
\label{defn: chars associated to lambda}
  Let $\lambda$ be an element of $(\bb{Z}^{n}_{+})^{\Hom(M,K)}$.  We
  associate to $\lambda$ an $n$-tuple of characters $ I_{M}
  \rightarrow \bigO^{\times}$ as follows. For $j=1,\ldots,n$ define
\begin{eqnarray*}
  \chi^{\lambda}_{j} : I_M & \rightarrow & \bigO^{\times} \\
  \sigma & \mapsto & \epsilon(\sigma)^{-(j-1)} \prod_{\tau : M
    \hookrightarrow K}
  \tau(\Art_{M}^{-1}(\sigma))^{-\lambda_{\tau,n-j+1}}.
\end{eqnarray*}
\end{defn}
Note that $\chi^{\lambda}_j$ can also be thought of as the restriction to $I_M$ of any crystalline character $G_M \rightarrow \Qlbar^{\times}$ whose Hodge-Tate weight with respect to $\tau : M \into \Qlbar$ is given by $(j-1)+\lambda_{\tau,n-j+1}$ for all $\tau$ (we use the convention that the Hodge-Tate
  weights of $\epsilon$ are all $-1$).

Let $\lambda$ be an element of $(\bb{Z}^{n}_{+})^{\Hom(M,K)}$.  We
associate to $\lambda$ an $l$-adic Hodge type $\mathbf{v}_{\lambda}$
in the sense of section 2.6 of \cite{kisinpst} as follows. Let $D_{K}$
denote the vector space $K^{n}$. Let $D_{K,M}=D_{K} \otimes_{\bb{Q}_l}
M$. For each embedding $\tau : M \hookrightarrow K$, we let
$D_{K,\tau}=D_{K,M} \otimes_{K\otimes M,1\otimes \tau}K $ so that
$D_{K,M} = \oplus_{\tau} D_{K,\tau}$. For each $\tau$ choose a
decreasing filtration $\Fil^{i}D_{K,\tau}$ of $D_{K,\tau}$ so that
$\dim_{K} \gr^{i}D_{K,\tau} = 0 $ unless $ i =
(j-1)+\lambda_{\tau,n-j+1}$ for some $j=1,\ldots,n$ in which case
$\dim_{K} \gr^{i}D_{K,\tau}=1$. We define a decreasing filtration of
$D_{K,M}$ by $K \otimes_{\bb{Q}_l} M$-submodules by setting
\[ \Fil^{i} D_{K,M} = \oplus_{\tau} \Fil^{i}D_{K,\tau}.\] Let
$\mathbf{v}_{\lambda}= \{ D_{K}, \Fil^{i}D_{K,M} \}$.

  Let $B$ denote a finite, local $K$-algebra and $\rho_B : G_{M}
  \rightarrow \GL_n(B)$ a crystalline representation. Then
  $D_B:=D_{\cris}(\rho_B)= (\rho_B \otimes_{\bb{Q}_l} B_{\cris})^{G_M}$ is a free $B\otimes_{\bb{Q}_l} M_0$-module of
  rank $n$ where $M_0$ is the maximal subfield of $M$ which is
  unramified over $\bb{Q}_l$. Moreover, $D_{B}$ is equipped with a
  $B$-linear and $\varphi_0$-semilinear morphism $\varphi_B$ where $\varphi_0$
  denotes geometric Frobenius on $M_0$. For each embedding $\tau : M_0
  \rightarrow K$, let $D_{B,\tau} = D_{B} \otimes_{B\otimes
    M_0,1\otimes \tau}B$. Then $D_{B,\tau} =
  \oplus_{\tau}D_{B,\tau}$. Also, for each $\tau$, $\varphi_B$ defines an
  isomorphism of $B$-modules $\varphi_B: D_{B,\tau} \isoto D_{B,\tau \circ
    \varphi_0^{-1}}$. Let
  $f=[M_0:\bb{Q}_l]$. Then $\varphi_B^f$ is a $B$-linear
  endomorphism of $D_B$ which preserves each $D_{B,\tau}$. For each
  $\tau$, the isomorphism $\varphi_B : D_{B,\tau} \rightarrow
  D_{B,\tau \circ \varphi_0^{-1}}$ takes $\phi_B^f$ to $\phi_B^f$. Let
  $P_{B}(X) \in B[X]$ denote the characteristic polynomial of
  $\phi_B^f$ on $D_{B,\tau}$ for any choice of $\tau$.

  Let $\mc{F}$ denote the flag variety over $\Spec \mc{O}$ whose set
  of $A$-points, for any $\mc{O}$-algebra $A$, corresponds to
  filtrations $0=\Fil_{0} \subset \Fil_1 \subset \ldots \subset \Fil_n
  = A^{n}$     
 of $A^{n}$ by locally free submodules which, locally, are
  direct summands and are such that $\Fil_{j}$ has rank $j$.
\end{notation}

\begin{defn}
  \label{defn:ordinary gal rep}
  Let $E$ be an algebraic extension of $K$ let $B$ be a finite local $E$-algebra. Let $\rho : G_{M} \rightarrow \GL_n(B)$ be a continuous
  homomorphism. We say that $\rho$ is \emph{ordinary of
    weight} $\lambda \in (\bb{Z}^{n}_{+})^{\Hom(M,K )}$ if $\rho$ is
  conjugate to a representation of the form
  \[ \left(\begin{matrix} \psi_{1} & * & \ldots & * &* \cr 0 &
      \psi_{2} & \ldots & * &* \cr

      \vdots & \vdots &\ddots & \vdots& \vdots \cr 0 & 0 & \ldots &
      \psi_{n-1} &*\cr 0 & 0 & \ldots & 0&
      \psi_{n}
    \end{matrix} \right)
  \]
  where for each $j = 1,\ldots,n$ the character
  $\psi_{j}$ agrees on an open subgroup of $I_M$ with
  the character $\chi^{\lambda}_{j}$ introduced above.
  
  Equivalently, $\rho$ is 
  ordinary of weight $\lambda$ if there is a full flag $0 = \Fil_{0}
  \subset \Fil_1 \subset \ldots \subset \Fil_{n} = B^{n}$ of $B^{n}$
  which preserved by $G_{M}$ and such that the representation of
  $G_{M}$ on $\gr_{j} = \Fil_{j}/\Fil_{j-1}$ is potentially semistable
  and for each embedding $\tau : M \hookrightarrow K$, the Hodge-Tate
  weight of $\gr_{j}$ with respect to $\tau$ is $(j-1) +
  \lambda_{\tau, n-j+1}$.
\end{defn}

\begin{lem}
  \label{OrdImpliesST}
  Suppose that $E$ is an algebraic extension of $K$ and $\rho : G_{M}
  \rightarrow \GL_n(E)$ is ordinary of weight $\lambda$. Let
  $\psi_{1},\ldots,\psi_{n} : G_{M} \rightarrow E^{\times}$ be as
  above. Then
  \begin{enumerate}
  \item $\rho$ is potentially semistable.
  \item If each $\psi_{j}$ is crystalline (which occurs if and only if
    $\psi_j$ agrees with $\chi^{\lambda}_j$ on all of $I_M$), then $\rho$ is semistable.
  \item If each $\psi_{j}$ is crystalline and $\lambda$ is regular,
    then $\rho$ is crystalline.
  \end{enumerate}
\end{lem}

\begin{proof} Part 2 follows from Proposition 1.28(2) of \cite{nekovar}
  and part 1 follows from part 2. Part 3 follows from Proposition 1.26
  of \cite{nekovar} and the formulae in
  Proposition 1.24 of \cite{nekovar}.
\end{proof}

\begin{lem}
  \label{lem:Ordinary always lifts in regular weight}  Let $\psi_i: G_M \rightarrow
  E^{\times}$ be as above (with respect to some $\lambda\in(\Z^n_+)^{\Hom(M,K)}$), with each
  $\psi_i$ crystalline. Suppose that $\rhobar:G_M\to\GL_n(k)$ is of the form \[ \left(\begin{matrix} \overline{\mu}_{1} & * & \ldots & * &* \cr 0 &
      \overline{\mu}_{2} & \ldots & * &* \cr

      \vdots & \vdots &\ddots & \vdots& \vdots \cr 0 & 0 & \ldots &
      \overline{\mu}_{n-1} &*\cr 0 & 0 & \ldots & 0&
      \overline{\mu}_{n}
    \end{matrix} \right)
  \]where $\overline{\psi}_i=\overline{\mu}_i$ for each $1\le i\le
  n$. Suppose that for each $i<j$ we have $\overline{\mu}_i\overline{\mu}_j^{-1}\neq\overline\epsilon$. Then $\rhobar$ has a lift to a crystalline representation
  $\rho:G_M\to\GL_n(E)$ of the form \[ \left(\begin{matrix} \psi_{1} & * & \ldots & * &* \cr 0 &
      \psi_{2} & \ldots & * &* \cr

      \vdots & \vdots &\ddots & \vdots& \vdots \cr 0 & 0 & \ldots &
      \psi_{n-1} &*\cr 0 & 0 & \ldots & 0&
      \psi_{n}
    \end{matrix} \right).
  \]
\end{lem}
\begin{proof}The fact that any upper-triangular representation of this
  form is crystalline follows easily as in the proof of Lemma
  \ref{OrdImpliesST}, because the assumption that
  $\overline{\mu}_i\overline{\mu}_j^{-1}\neq \epsilon$ implies that
  $\psi_i\psi_j^{-1}\neq \epsilon$. The fact that such an upper-triangular
  lift exists follows from the fact that $H^2(G_M,\mathfrak{u})=0$,
  where $\mathfrak{u}$ is the subspace of the Lie algebra $\ad\rhobar$
  consisting of strictly upper-triangular matrices. The vanishing of
  this cohomology group follows from Tate local duality and the
  existence of a filtration on $\mathfrak{u}$ whose graded pieces
  are one-dimensional with $G_M$ acting via the characters
  $\overline{\mu}_i\overline{\mu}_j^{-1}\neq \epsilon$, $i<j$ (cf. Lemma 3.2.3 of \cite{ger}).
\end{proof}

We now recall some results of Kisin. Let $\lambda$ be an element of $(\bb{Z}^{n}_{+})^{\Hom(M,K)}$ and let
$\mathbf{v}_{\lambda}$ be the associated $l$-adic Hodge type.
\begin{defn}
  If $B$ is a finite $K$-algebra and $V_{B}$ is a free $B$-module of
  rank $n$ with a continuous action of $G_{M}$ that makes $V_B$ into a
  de Rham representation, then we say that \emph{$V_B$ is of $l$-adic Hodge
  type $\mathbf{v}_{\lambda}$} if for each $i$ there is an isomorphism
  of $B \otimes_{\bb{Q}_l} M$-modules
\[ \gr^{i}(V_{B} \otimes_{\bb{Q}_l} B_{dR})^{G_{M}}
\isoto B \otimes_K (\gr^{i}D_{K,M}). \] 
\end{defn}
For example, if $E$ is a finite extension of $K$ and $\rho : G_M \rightarrow \GL_n(E)$ is ordinary of
weight $\lambda$, then $\rho$ is of $l$-adic Hodge type $\mathbf{v}_{\lambda}$.

Corollary 2.7.7 of \cite{kisinpst} implies that there is a unique
$l$-torsion-free quotient $R^{\mathbf{v}_{\lambda},cr}_{\rhobar}$ of
$R^{\square}_{\rhobar}$ with the property that for any finite
$K$-algebra $B$, a homomorphism of $\mc{O}$-algebras $\zeta :
R^{\square}_{\rhobar}\rightarrow B$ factors through
$R^{\mathbf{v}_{\lambda},cr}_{\rhobar}$ if and only if $\zeta \circ
\rho^{\square}$ is crystalline of $l$-adic Hodge type
$\mathbf{v}_{\lambda}$. Moreover, Theorem 3.3.8 of \cite{kisinpst}
implies that $\Spec R^{\mathbf{v}_{\lambda},cr}_{\rhobar}[1/l]$ is
formally smooth over $K$ and equidimensional of dimension $n^2 +
\frac{1}{2}n(n-1)[M:\bb{Q}_l]$. 

Let $\mc{F}$ be the flag
variety over $\Spec \mc{O}$ as introduced above and let
$\mc{G}^{\lambda}$ be the closed subscheme of $\mc{F}\times_{\Spec
  \mc{O}} \Spec R^{\mathbf{v}_{\lambda},cr}_{\rhobar}$ corresponding
to filtrations $\Fil$ which (i) are preserved by the induced action of
$G_M$ and (ii) are such that $I_M$ acts on $\gr_j=\Fil_j/\Fil_{j-1}$
via the character $\chi^{\lambda}_{j}$ for each $j=1,\ldots,n$. The
fact that $\mc{G}^{\lambda}$ is a closed subscheme can proved in the
same way as Lemma 3.1.2 of \cite{ger}. Let
$R^{\triangle_{\lambda},cr}_{\rhobar}$ be the image of 
\[ R^{\mathbf{v}_{\lambda},cr}_{\rhobar} \rightarrow
\mc{O}_{\mc{G}^{\lambda}}(\mc{G}^{\lambda}[1/l]). \]
In other words, $\Spec R^{\triangle_{\lambda},cr}_{\rhobar}$ is the
scheme theoretic image of the morphism
$\mc{G}^{\lambda}[1/l] \rightarrow \Spec
R^{\mathbf{v}_{\lambda},cr}_{\rhobar}$. The next result follows from Lemma 3.3.3 of \cite{ger}.

\begin{lem}
\label{lem:ordinary crystalline ring for GL_n}
  For any finite local $K$-algebra $B$, a homomorphism of $\mc{O}$-algebras $\zeta :
  R^{\mathbf{v}_{\lambda},cr}_{\rhobar} \rightarrow B$ factors through
  $R^{\triangle_{\lambda},cr}_{\rhobar}$ if and only if $\zeta \circ
  \rho^{\square}$ is ordinary of weight $\lambda$. Moreover,  $\Spec R^{\triangle_{\lambda},cr}_{\rhobar}$ is a union of irreducible components of
  $\Spec R^{\mathbf{v}_{\lambda},cr}_{\rhobar}$.
\end{lem}

\subsubsection{The $p=l$ case with a slight refinement}\label{Refined
  liftings for l=p}

We continue to consider, as above, crystalline lifts of $\rhobar$
which are ordinary of a given weight $\lambda$. A necessary condition for
such lifts to exist is that  $\rhobar$ itself is conjugate to an upper
triangular representation whose ordered $n$-tuple of diagonal
characters, restricted to $I_M$, is given by
$(\overline{\chi}^{\lambda}_1,\ldots,\overline{\chi}^{\lambda}_n)$. Let
us assume that $\rhobar$ has this property. In fact, let us  
fix characters
$\overline{\mu}_{1},\ldots,\overline{\mu}_{n} : G_{M} \rightarrow k^{\times}$ with
$\overline{\mu}_{j}|_{I_M} = \overline{\chi}_{j}^{\lambda}$ and  assume that $\rhobar$ is
conjugate to an upper triangular representation whose ordered
$n$-tuple of diagonal characters is $\mubar:=(\mubar_{1},\ldots,\mubar_n)$. Note that
if the characters $\overline{\chi}^{\lambda}_j$ are not distinct, then
we may have more than one choice for the ordered $n$-tuple
$(\mubar_1,\ldots,\mubar_n)$. We now would like to study crystalline lifts of
$\rhobar$ which are ordinary of weight $\lambda$ \emph{and} are such
that for each $j$, the character $\psi_j$ of Definition \ref{defn:ordinary
  gal rep} lifts $\mubar_j$.

Let $R_{\mubar}$ denote the object of $\mc{C}_{\mc{O}}$ representing
the functor which sends an object $A$ of $\mc{C}_{\mc{O}}$ to the set
of lifts $(\psi_1,\ldots,\psi_n)$ of the ordered $n$-tuple
$(\mubar_1,\ldots,\mubar_n)$ with $\psi_j|_{I_M}=\chi^{\lambda}_{j}$ for
each $j$. The ring $R_{\mubar}$ is non-canonically isomorphic to
$\mc{O}[[X_1,\ldots,X_n]]$. Let $(\psi^{\univ}_1,\ldots,\psi^{\univ}_n)$
be the universal lift of the tuple $(\mubar_1,\ldots,\mubar_n)$ to
$R_{\mubar}$. Let $\mc{G}^{\lambda}_{\mubar}$ denote the closed
subscheme of the flag variety $\mc{F}\times_{\Spec \mc{O}} \Spec
(R^{\triangle_{\lambda},cr}_{\rhobar}\widehat{\otimes}_{\mc{O}}R_{\mubar})$
corresponding to filtrations which are (i) preserved by the induced
action of $G_M$ and (ii) such that $G_M$ acts on $\gr_j$ via the
pushforward of $\psi^{\univ}_j$ for each $j=1,\ldots,n$. Let
$R^{\triangle_{\lambda},cr}_{\rhobar,\mubar}$ be the quotient of $
R^{\triangle_{\lambda},cr}_{\rhobar}\widehat{\otimes}_{\mc{O}}R_{\mubar}$
corresponding to the scheme theoretic image of
$\mc{G}^{\lambda}_{\mubar}[1/l]$.  Note that we have a natural
morphism $\mc{G}^{\lambda}_{\mubar}[1/l] \rightarrow
\mc{G}^{\lambda}[1/l]$ covering the morphism $\Spec
R^{\triangle_{\lambda},cr}_{\rhobar,\mubar} \rightarrow \Spec
R^{\triangle_{\lambda},cr}_{\rhobar}$.

\begin{lem}\label{lem: picking components of the ordinary crystalline
    deformation ring}
  After inverting $l$, the morphism $\Spec R^{\triangle_{\lambda},cr}_{\rhobar,\mubar}
  \rightarrow \Spec R^{\triangle_{\lambda},cr}_{\rhobar}$ becomes a closed
  immersion and identifies $\Spec R^{\triangle_{\lambda},cr}_{\rhobar,\mubar}[1/l]$ with a union of irreducible
  components of $\Spec R^{\triangle_{\lambda},cr}_{\rhobar}[1/l]$. Moroever, every irreducible component of $\Spec R^{\triangle_{\lambda},cr}_{\rhobar}$ arises in this way. 
\end{lem}

\begin{proof}
  Let $X^{ord,cr}=\Spec R^{\triangle_{\lambda},cr}_{\rhobar}$ and let
  $X^{ord,cr}_{\mubar} = \Spec
  R^{\triangle_{\lambda},cr}_{\rhobar,\mubar}$. Let $x$ be a closed
  point of $X^{ord,cr}_{\mubar}[1/l]$ with residue field $E$. Let $z$ denote the image of $x$ in $X^{ord,cr}$. We claim
  that the natural map on completed local rings
  $\mc{O}_{X^{ord,cr},z}^{\wedge} \rightarrow
  \mc{O}_{X^{ord,cr}_{\mubar},x}^{\wedge}$ is an isomorphism.

  With this in mind, let $B$ denote a finite, local $E$-algebra and $\zeta :
  \mc{O}_{X^{ord,cr},z}^{\wedge} \rightarrow B$ an $E$-algebra
  homomorphism. It follows from Lemma \ref{lem:ordinary crystalline ring
    for GL_n} that $\zeta$ corresponds to a crystalline
  representation $\rho_B : G_M \rightarrow GL_n(B)$ which preserves a
  full flag $0\subset \Fil_1 \subset \ldots \subset \Fil_n = B^n$ with
  $I_M$ acting on $\gr_j$ via $\chi^{\lambda}_j$. Since the characters
  $\chi^{\lambda}_j$ are pairwise distinct, there is a unique such
  flag. Moreover, there exists a subring $A \subset B$ which is local
  and finite over $\mc{O}_E$ and such that the action of $G_M$ on $\gr_j$
  is given by a character
  $\psi_j: G_M \rightarrow B^{\times}$ which factors through
  $A^{\times}$ with $\psi_{j} \mod \mf{m}_{A} = \mubar_j
  \otimes_k A/\mf{m}_A$.
We see that there is a unique lifting of $\zeta^{*} : \Spec B
  \rightarrow \Spec R^{\triangle_{\lambda},cr}_{\rhobar}$ to
  $\mc{G}^{\lambda}_{\mubar}$. It follows that the homomorphism
  $\mc{O}_{X^{ord,cr},z}^{\wedge} \rightarrow
  \mc{O}_{X^{ord,cr}_{\mubar},x}^{\wedge}$ is formally smooth of
  relative dimension 0. It's easy to see that both sides have the same
  residue field and hence the map is an isomorphism.

  For each closed point in $X^{ord,cr}[1/l]$, there is at
  most 1 closed point of $X^{ord,cr}_{\mubar}[1/l]$ lying over it. From this,
  and the claim just established, 
    we deduce the first two statements.
  The final statement is clear.
\end{proof}

\subsubsection{The $p=l$ case in non-fixed weight}
\label{sec: p=l case in non-fixed weight}

In this section we assume that $p=l$, that $\rhobar : G_M \rightarrow
\GL_n(k)$ is the trivial homomorphism. Let $R_{\rhobar}^{\square}$
denote the universal $\mc{O}$-lifting ring of $\rhobar$ and let
$\Lambda_M = \mc{O}[[I_{M^{\ab}}(l)^n]]$ where for a group $H$, $H(l)$ denotes its pro-$l$ completion. Then $\Lambda_M$ represents the functor $\mc{C}_{\mc{O}} \rightarrow Sets$ sending an algebra $A$ to the set of ordered $n$-tuples $(\chi_1,\ldots,\chi_n)$ of characters $\chi_j : I_{M^{\ab}} \rightarrow A^{\times}$ lifting the trivial character modulo $\mf{m}_A$. Let $\rho^{\square}$ denote the universal lift of $\rhobar$ to $R^{\square}_{\rhobar}$ and let $(\chi_1^{\univ},\ldots,\chi_n^{\univ})$ denote the universal $n$-tuple of characters $I_{M^{\ab}} \rightarrow \Lambda^{\times}_M$.

Let $R^{\square}_{\rhobar,\Lambda_M} = R^{\square}_{\rhobar}
\widehat{\otimes}_{\mc{O}} \Lambda_M$. Let $\mc{G}$ denote the
closed subscheme of the flag variety $\mc{F} \times_{\Spec \mc{O}}
\Spec R^{\square}_{\rhobar,\Lambda_M}$ corresponding to filtrations
which are (i) preserved by the induced action of $G_M$ and (ii) such
that $I_M$ acts on $\gr_j$ via the pushforward of
$\chi^{\univ}_{j}$. Let $R^{\triangle}_{\rhobar,\Lambda_M}$ be the
quotient of $R^{\square}_{\rhobar,\Lambda_M}$ corresponding to the
scheme theoretic image of the morphism
\[ \mc{G}[1/l] \rightarrow \Spec R^{\square}_{\rhobar,\Lambda_M} .\]
If $E$ is a finite extension of $K$, a homomorphism of $\mc{O}$-algebras $\zeta :
R^{\square}_{\rhobar,\Lambda_M} \rightarrow E$ factors through
$R^{\triangle}_{\rhobar,\Lambda_M}$ if and only if $\zeta \circ
\rho^{\square}$ is conjugate to an upper triangular representation
whose ordered $n$-tuple of diagonal characters, restricted to $I_M$, is the pushforward of
$(\chi_1^{\univ},\ldots,\chi_n^{\univ})$.

\subsection{Global deformation rings}
\label{subsec:global deformation rings}
\subsubsection{The group $\mathcal{G}_n$}
Let $n$ be a positive integer, and let $\Gn$ be the group scheme over
$\Z$ which is the semidirect product of $\GL_n\times\GL_1$ by the
group $\{1,j\}$, which acts on $\GL_n\times\GL_1$
by \[j(g,\mu)j^{-1}=(\mu{}^tg^{-1},\mu).\] There is a homomorphism
$\nu:\Gn\to\GL_1$ sending $(g,\mu)$ to $\mu$ and $j$ to $-1$. Write
$\mathfrak{g}_n^0$ for the trace zero subspace of the Lie algebra of
$\GL_n$, regarded as a Lie subalgebra of the Lie algebra of $\Gn$.

\begin{defn}
  Let $F^+$ be a totally real field, and let $r:G_{F^+}\to\Gn(L)$ be a
  continuous homomorphism, where $L$ is a topological field. Then we
  say that $r$ is \emph{odd} if for all complex conjugations $c_v\in
  G_{F^+}$, $\nu\circ r(c_v)=-1$.
\end{defn}

\subsubsection{Bigness}Recall definition 2.5.1 of \cite{cht}.
\begin{defn}

  Let $k$ be an algebraic extension of the finite field $\F_l$. We say
  that a finite subgroup $H\subset \GL_n(k)$ is \emph{big} if the
  following conditions are satisfied.
  \begin{itemize}
  \item $H$ has no quotient of $l$-power order.
  \item $H^0(H,\mathfrak{g}_n^0(k))=(0)$.
  \item $H^1(H,\mathfrak{g}_n^0(k))=(0)$.
  \item For all irreducible $k[H]$-submodules $W$ of
    $\mathfrak{g}_n^0(k)$ we can find $h\in H$ and $\alpha\in k$ such
    that the $\alpha$-generalised eigenspace $V_{h,\alpha}$ of $h$ in
    $k^n$ is one-dimensional and furthermore $\pi_{h,\alpha}\circ
    W\circ i_{h,\alpha}\neq (0)$. Here $\pi_{h,\alpha}:k^n\to
    V_{h,\alpha}$ is the $h$-equivariant projection of $k^n$ to
    $V_{H,\alpha}$, and $i_{h,\alpha}$ is the $h$-equivariant
    injection of $V_{h,\alpha}$ into $k^n$.
  \end{itemize}
We call a finite subgroup $H\subset\Gn(k)$ \emph{big} if $H$ surjects
onto $\Gn(k)/\Gn^0(k)$ and $H\cap\Gn^0(k)$ is big.
\end{defn}

\subsubsection{Deformation problems}
\label{sec: deformation problems}

Let $F/F^{+}$ be a totally imaginary quadratic extension of a totally
real field $F^+$. Let $c$ denote the non-trivial element of $\Gal(F/F^+)$. Let $k$ denote a finite field of characteristic $l$ and $K$ a
finite extension of $\bb{Q}_l$, inside a fixed algebraic closure $\Qlbar$, with ring of integers $\mc{O}$ and
residue field $k$. Assume that $K$ contains the image of every
embedding $F \into \Qlbar$ and that the prime $l$ is odd. Assume that
every place in $F^+$ dividing $l$ splits in $F$. Let $S$ denote a finite set of finite places of $F^+$ which split in
$F$, and assume that $S$ contains every place dividing $l$. Let $S_l$ denote the set of places of $F^+$ lying over $l$. Let $F(S)$
denote the maximal extension of $F$ unramified away from $S$. Let
$G_{F^{+},S}=\Gal(F(S)/F^{+})$ and $G_{F,S}=\Gal(F(S)/F)$. For each $v
\in S$ choose a place $\wt{v}$ of $F$ lying over $v$ and let $\wt{S}$
denote the set of $\wt{v}$ for $v \in S$. For each place $v|\infty$ of $F^+$ we let
$c_v$ denote a choice of a complex conjugation at $v$ in
$G_{F^+,S}$.  For each place $w$ of $F$
we have a $G_{F,S}$-conjugacy class of homomorphisms $G_{F_w}
\rightarrow G_{F,S}$. For $v \in S$ we fix a choice of homomorphism
$G_{F_{\wt{v}}} \rightarrow G_{F,S}$.
 
If $R$ is a ring and $r : G_{F^+,S} \rightarrow \mc{G}_n(R)$
is a homomorphism with $r^{-1}(\GL_n(R)\times \GL_1(R))=G_{F,S}$, we
will make a slight abuse of notation and write $r|_{G_{F,S}}$
(respectively $r|_{G_{F_w}}$ for $w$ a place of $F$) to mean
$r|_{G_{F}}$ (respectively $r|_{G_{F_w}}$) composed with the
projection $\GL_n(R)\times \GL_1(R) \rightarrow \GL_n(R)$.

Fix a continuous homomorphism
\[ \rbar : G_{F^+,S} \rightarrow \mc{G}_n(k) \] such that $G_{F,S} =
\rbar^{-1}(\GL_n(k)\times \GL_1(k))$ and fix a continuous character
$\chi : G_{F^+,S}\rightarrow \mc{O}^{\times} $ such that $\nu \circ
\rbar = \overline{\chi}$. Assume that $\rbar|_{G_{F,S}}$ is absolutely
irreducible. 
As in Definition 1.2.1 of \cite{cht}, we define
\begin{itemize}
\item a \emph{lifting} of $\rbar$ to an object $A$ of
  $\mc{C}_{\mc{O}}$ to be a continuous homomorphism $r : G_{F^+,S}
  \rightarrow \mc{G}_n(A)$ lifting $\rbar$ and with $\nu \circ r =
  \chi$;
\item two liftings $r$, $r^{\prime}$ of $\rbar$ to $A$ to be
  \emph{equivalent} if they are conjugate by an element of
  $\ker(\GL_n(A)\rightarrow \GL_n(k))$;
\item a \emph{deformation} of $\rbar$ to an object $A$ of
  $\mc{C}_{\mc{O}}$ to be an equivalence class of liftings.
\end{itemize}

For each place $v \in S$, let $R^{\square}_{\rbarwtv}$ denote
the universal $\mc{O}$-lifting ring of $\rbar|_{G_{F_{\wt{v}}}}$ and
let $R_{\wt{v}}$ denote a quotient of $R^{\square}_{\rbarwtv}$ which satisfies the
following property: 
\begin{itemize}
\item[(*)] let $A$ be an object of $\mc{C}_{\mc{O}}$ and let
$\zeta,\zeta^{\prime}: R^{\square}_{\rbarwtv}\rightarrow A$ be homomorphisms corresponding to
two lifts $r$ and $r^{\prime}$ of $\rbarwtv$ which are
conjugate by an element of $\ker(\GL_n(A)\rightarrow
\GL_n(k))$. Then $\zeta$ factors through $R_{\wt{v}}$ if and only if
$\zeta^{\prime}$ does.
\end{itemize}
We
consider the \emph{deformation problem}
\[ \mc{S} = (F/F^{+},S,\wt{S},\mc{O},\rbar,\chi,\{ R_{\wt{v}}\}_{v \in
  S}) \]
(see sections 2.2 and 2.3 of \cite{cht} for this terminology).
We say that a lifting $r : G_{F^+,S}\rightarrow \mc{G}_n(A)$ is
\emph{of type $\mc{S}$} if for each place $v \in S$, the homomorphism
$R^{\square}_{\rbarwtv} \rightarrow A$ corresponding to $r|_{G_{F_{\wt{v}}}}$ factors
through $R_{\wt{v}}$. We also define deformations of type $\mc{S}$ in the same way.

Let $\Def_{\mc{S}}$ be the functor $ \mc{C}_{\mc{O}}\rightarrow Sets$
which sends an algebra $A$ to the set of deformations of $\rbar$ to
$A$ of type $\mc{S}$. 
By Proposition 2.2.9 of
\cite{cht} this functor is represented by an object
$R_{\mc{S}}^{\univ}$ 
of $\mc{C}_{\mc{O}}$.

\begin{lem}
  \label{lem: components are conjugation invariant}
  Let $M$ be a finite extension of $\bb{Q}_p$ for some prime $p$ and
  $\rhobar : G_M \rightarrow \GL_n(k)$ a continuous homomorphism. If
  $p \neq l$, let $\tau$ be an inertial type for $G_M$ over $K$ and
  let $R$ be a quotient of $R^{\square,\tau}_{\rhobar}$ corresponding
  to a union of irreducible components. If
  $p=l$, assume that $K$ contains the image of every embedding $M
  \into \overline{K}$, let $\lambda \in
  (\bb{Z}^n_+)^{\Hom(M,K)}$ and let $R$ be a quotient of
  $R^{\mathbf{v}_{\lambda},cr}_{\rhobar}$ corresponding to a union of
  irreducible components. Then $R$ satisfies property (*) above.
\end{lem}

\begin{proof}
  We consider the case $p = l$; the other case is similar. Let
  $R^{\mathbf{v}_{\lambda},cr}_{\rhobar}[[\underline{X}]]=R^{\mathbf{v}_{\lambda},cr}_{\rhobar}[[X_{ij}:1\leq
  i,j \leq n]]$ and consider the lift of $\rhobar$ to
  $R^{\mathbf{v}_{\lambda},cr}_{\rhobar}[[\underline{X}]]$ given by
  $(1_n+(X_{ij}))\rho^{\square}(1_n+(X_{ij}))^{-1}$. This lift gives
  rise to an $\mc{O}$-algebra homomorphism $R^{\square}_{\rhobar}
  \rightarrow
  R^{\mathbf{v}_{\lambda},cr}_{\rhobar}[[\underline{X}]]$. We claim
  that this homomorphism factors through
  $R^{\mathbf{v}_{\lambda},cr}_{\rhobar}$. This follows from the fact
  that $R^{\mathbf{v}_{\lambda},cr}_{\rhobar}[[\underline{X}]]$ is
  reduced and $l$-torsion-free and every $\Qlbar$ point of this ring
  gives rise to a lift of $\rhobar$ which is crystalline of $l$-adic
  Hodge type $\mathbf{v}_{\lambda}$. Let $\alpha$ denote the resulting
  $\mc{O}$-algebra homomorphism $R^{\mathbf{v}_{\lambda},cr}_{\rhobar}
  \to R^{\mathbf{v}_{\lambda},cr}_{\rhobar}[[\underline{X}]]$ and let
  $\iota : R^{\mathbf{v}_{\lambda},cr}_{\rhobar} \to
  R^{\mathbf{v}_{\lambda},cr}_{\rhobar}[[\underline{X}]]$ denote the
  standard $R^{\mathbf{v}_{\lambda},cr}_{\rhobar}$-algebra structure
  on $R^{\mathbf{v}_{\lambda},cr}_{\rhobar}[[\underline{X}]]$.

  The irreducible components of $\Spec
  R^{\mathbf{v}_{\lambda},cr}_{\rhobar}[[\underline{X}]]$ and $\Spec
  R^{\mathbf{v}_{\lambda},cr}_{\rhobar}$ are in natural bijection (if
  $\wp$ is a minimal prime of $R^{\mathbf{v}_{\lambda},cr}_{\rhobar}$,
  then $\iota(\wp)$ generates a minimal prime of
  $R^{\mathbf{v}_{\lambda},cr}_{\rhobar}[[\underline{X}]]$). Let $\wp$
  be a minimal prime of $R^{\mathbf{v}_{\lambda},cr}_{\rhobar}$. We
  claim that the kernel of the map $\beta :
  R^{\mathbf{v}_{\lambda},cr}_{\rhobar} \to
  R^{\mathbf{v}_{\lambda},cr}_{\rhobar}[[\underline{X}]]/\iota(\wp) =
  (R^{\mathbf{v}_{\lambda},cr}_{\rhobar}/\wp)[[\underline{X}]]$
  induced by $\alpha$ is $\wp$. To see this note that the map $\gamma:
  R^{\mathbf{v}_{\lambda},cr}_{\rhobar}[[\underline{X}]] \onto
  R^{\mathbf{v}_{\lambda},cr}_{\rhobar}[[\underline{X}]]/(X_{ij})
  \cong R^{\mathbf{v}_{\lambda},cr}_{\rhobar}$ satisfies $\gamma \circ
  \alpha = \id$. From this it follows that $\ker \beta \subset
  \wp$. Since $\wp$ is minimal, we must have $\ker \beta = \wp$. If
  $\wp_1,\ldots,\wp_k$ are minimal primes of
  $R^{\mathbf{v}_{\lambda},cr}_{\rhobar}$ and $I=\wp_1 \cap \dots \cap
  \wp_k$, we deduce that the kernel of the map
  $R^{\mathbf{v}_{\lambda},cr}_{\rhobar} \stackrel{\alpha}{\to}
  (R^{\mathbf{v}_{\lambda},cr}_{\rhobar}/I)[[\underline{X}]]$ is
  $I$. The lemma follows.
\end{proof}

\subsubsection{A lower bound}
Let $F,F^{+},S, \wt{S}$ and $\rbar$ be as in the previous section. In this section we will give a lower bound on the Krull dimension of the ring $R^{\univ}_{\mc{S}}$ for certain deformation problems $\mc{S}$.

For each place $v \in S$ away from $l$, fix an inertial type $\tau_{v}$ for $I_{F_{\wt{v}}}$ and assume that $\rbarwtv$ has a lift of type $\tau_v$ (in other words, $R^{\square,\tau_{v}}_{\rbarwtv}$ is non-zero). Let $R_{\wt{v}}$ be a quotient of $R^{\square,\tau_{v}}_{\rbarwtv}$ corresponding to a union of irreducible components.

For each place $v \in S$ lying above $l$, let $\lambda_{\wt{v}}$ be an
element of $(\bb{Z}^{n}_+)^{\Hom(F_{\wt{v}},K)}$, and assume that
$\rbarwtv$ has a crystalline lift which is ordinary of weight
$\lambda_{\wt{v}}$ and let $R_{\wt{v}}$ be a quotient of the ring
$R^{\triangle_{\lambda_{\wt{v}}},cr}_{\rbarwtv}$ corresponding to a
union of irreducible components. Let
\[ \mc{S}  = (F/F^{+},S,\wt{S},\mc{O},\rbar,\chi,\{ R_{\wt{v}}\}_{v \in
  S}). \]

\begin{lem}\label{lem: lower bound on unitary dimension} Assume that $\rbar$ is odd, and that 
   $H^0(G_{F^+,S},\ad \rbar(1))=\{0\}$.  For
  $\mc{S}$ as above, the Krull dimension of $R^{univ}_{\mc{S}}$ is at
  least $1$.
\end{lem}
\begin{proof}
  By Corollary 2.3.5 of \cite{cht} (noting that $\chi(c_v)=-1$ for all
  $v|\infty$) we see that this dimension is at least \[1+\sum_{v\in
    S}(\dim R_{\wt{v}}-n^2-1)-\dim_k
  H^0(G_{F^+,S},\ad\rbar(1))-\sum_{v|\infty}n(n-1)/2.\]  For $v \in S$ away from $l$, we
  have $\dim R_{\wt{v}} = n^2+1$ by Lemma \ref{lem: local dimension
    for l not p}. For $v \in S$ lying over $l$ we have $\dim
  R_{\wt{v}} = n^2 + 1 + \frac{1}{2}n(n-1)[F_{v}^{+}:\Q_l]$ by Lemma
  \ref{lem:ordinary crystalline ring for GL_n} and the remark
  preceding it. 
 We therefore have
  \begin{align*}
    \sum_{v\in S}(\dim R_{\wt{v}}-n^2-1)
    &=\sum_{v|l}\frac{1}{2}n(n-1)[F^+_v:\Q_l] \\
    &=\frac{1}{2}n(n-1)[F^+:\Q]\\
    &=\sum_{v|\infty}n(n-1)/2, 
  \end{align*}giving
  the required bound.
\end{proof}

\subsubsection{A finiteness result}

Let $F,F^{+},S, \wt{S}$ and $\rbar$ be as in the previous two
sections.  Suppose that $L^+/F^+$ is a finite totally real
extension. Let $L=L^+F$. Let $S^{\prime}$ (resp.\ $\wt{S}^{\prime}$)
denote a set of places of $L^+$ (resp.\ $L$) all of which split in
$L$, containing all places lying over a place in $S$ (resp. containing
exactly one place above each place in $S^{\prime}$, and containing
every place lying above a place in $\wt{S}$). Let
$G_{L^+,S^{\prime}}=\Gal(L(S^{\prime})/L^+)$, where $L(S^{\prime})$ is
the maximal extension of $L$ unramified outside $S^{\prime}$. Let
$G_{L,S^{\prime}} = \Gal(L(S^{\prime})/L)$. We assume that
$\rbar|_{G_{L,S^{\prime}}}$ is absolutely irreducible.

Let
\[ \mc{S}_{0} = (F/F^{+},S,\wt{S},\mc{O},\rbar,\chi,\{ R_{\rbarwtv}^{\square}\}_{v \in
  S}) \]
and
\[ \mc{S}^{\prime}_{0} =  
(L/L^{+},S^{\prime},\wt{S}^{\prime},\mc{O},\rbar|_{G_{L^+,S^{\prime}}},\chi|_{G_{L^+,S^{\prime}}},\{
R_{\rbar|_{G_{L_{\wt{v}^{\prime}}}}}^{\square}\}_{v^{\prime} \in
  S^{\prime}}) \]
and let $R^{\univ}_{\mc{S}_0}$ and $R^{\univ}_{\mc{S}^{\prime}_0}$ denote
the rings representing the functors $\Def_{\mc{S}_0}$ and $\Def_{\mc{S}^{\prime}_0}$.
Restricting the universal
deformation valued in $R^{\univ}_{\mc{S}_0}$ to $G_{L^+,S^{\prime}}$
gives $R^{\univ}_{\mc{S}_0}$ the structure of a
$R^{\univ}_{\mc{S}^{\prime}_0}$-algebra.

\begin{lem}\label{lem: deformation ring is finite over another for the
  unitary group}
  $R^{\univ}_{\mc{S}_0}$ is a finite
  $R_{\mc{S}^{\prime}_0}^{\univ}$-algebra.
\end{lem}
\begin{proof}
  The argument is extremely similar to that of Lemma 3.6 of
  \cite{kw}. Write $\mathfrak{m}_{L^+}$ for the maximal ideal of
  $R_{\mc{S}^{\prime}_0}^{\univ}$, and let $r_{F^+,L^+}$ denote the
  $R^{\univ}_{\mc{S}_0}/\mathfrak{m}_{L^+}R^{\univ}_{\mc{S}_0}$-representation
  of $G_{F^+,S}$ obtained from the universal representation over
  $R^{\univ}_{\mc{S}_0}$.  By definition,
  $r_{F^+,L^+}|_{G_{L^+,S^{\prime}}}$ is equivalent to
  $\rbar|_{G_{L^+,S^{\prime}}}$. As a consequence, if $M$ denotes the
  normal closure of the composite of $L^+$ and the fixed field of
  $\ker\rbar$, then $r_{F^+,L^+}$ factors through $\Gal(M/F^+)$, and
  the image of $r_{F^+,L^+}$ is necessarily finite.

  Now, by Lemma 2.1.12 of \cite{cht}, we see that
  $R^{\univ}_{\mc{S}_0}/\mathfrak{m}_{L^+}R^{\univ}_{\mc{S}_0}$ is
  generated by the traces of the $r_{L^+,F^+}(g)$ for $g\in
  G_{F,S}$. Consider a prime ideal $\mathfrak{p}$ of
  $\overline{R}^{\univ}:=R^{\univ}_{\mc{S}_0}/\mathfrak{m}_{L^+}R^{\univ}_{\mc{S}_0}$. Because
  the image of $r_{F^+,L^+}$ is finite, we see that the images of
  these traces in $\overline{R}^{\univ}/\mathfrak{p}$ are sums of roots
  of unity of bounded degree, so that
  $\overline{R}^{\univ}/\mathfrak{p}$ is finite. Thus
  $\overline{R}^{\univ}$ is a 0-dimensional Noetherian ring, so it is
  Artinian, and thus a direct product of Artin local rings with finite
  residue fields. Thus $\overline{R}^{\univ}$ is finite. It follows
  that $R^{\univ}_{\mc{S}_0}$ is a finitely generated module over the
  complete local ring $R^{\univ}_{\mc{S}^{\prime}_0}$.
\end{proof}

\section{Ordinary automorphic representations}\label{sec:ordinary
  automorphic representations}

\subsection{Ordinary automorphic representations of $\GL_n$}
\label{subsec: ordinary automorphic representations}
Let $L$ be either a totally real number field or a quadratic totally imaginary
extension of a totally real number field.  Let $\lambda \in
(\bb{Z}^{n}_+)^{\Hom(L,\bb{C})}$.  Let $\pi$ be an automorphic
representation of $GL_n(\bb{A}_L)$ which is
\begin{itemize}
\item RAESDC (regular, algebraic, essentially-self-dual, cuspidal) of
  weight $\lambda$ if $L$ is totally real, or
\item RACSDC (regular, algebraic, conjugate-self-dual, cuspidal) of
  weight $\lambda$ if $L$ is totally imaginary.
\end{itemize}
See section 5 of \cite{tay06} or section 4 of \cite{cht} for
definitions of these terms. Let $l$ be a prime number and $\iota :
\overline{\bb{Q}}_l \isoto \bb{C}$ an isomorphism of
fields. Let $v$ be a place of $L$ dividing $l$ and $\varpi_{v}$ a
uniformizer in $\mc{O}_{L_v}$. For each $b > 0$, let $\Iw(v^{b,b})$
denote the open compact subgroup of $GL_n(\mc{O}_{L_v})$ consisting of
matrices which reduce modulo $v^{b}$ to a unipotent upper triangular
matrix. The space $(\iota^{-1}\pi_v)^{\Iw(v^{b,b})}$ carries commuting
actions of the scaled Hecke operators
\[ U_{\iota^*\lambda,\varpi_v}^{(j)} := \left(\prod_{i=1}^j \prod_{ \tau : L_v
    \hookrightarrow \overline{\bb{Q}}_l} \tau(\varpi_{v})^{- \lambda_{
      \iota \tau|_{L}, n-i+1}} \right) \left[ \Iw(v^{b,b})
  \left( \begin{matrix} \varpi_{v}1_j & 0 \cr 0 & 1_{n-j} \end{matrix}
  \right) \Iw(v^{b,b}) \right]
\]
for $j=1,\ldots,n$. We define the ordinary part $(\iota^{-1}
\pi_v)^{\Iw(v^{b,b}),\ord}$ of $(\iota^{-1}\pi_v)^{\Iw(v^{b,b})}$ to be
the maximal subspace which is invariant under each
$U_{\iota^*\lambda,\varpi_v}^{(j)}$ and such that every eigenvalue of each
$U_{\iota^*\lambda,\varpi_v}^{(j)}$ is an $l$-adic unit. We define
\[ (\iota^{-1}\pi_v)^{\ord} : = \varinjlim_{b>0}
(\iota^{-1}\pi_v)^{\Iw(v^{b,b}),\ord} .\] We say that $\pi$ is
$\iota$-ordinary at $v$ if the space
$(\iota^{-1}\pi_v)^{\ord}$ is non-zero.

\subsection{$l$-adic automorphic forms on definite unitary
  groups}\label{subsec:auto forms on unitary groups}

Let $F^{+}$ denote a totally real number field and $n$ a positive
integer. Let $F/F^{+}$ be a totally imaginary quadratic extension of
$F^{+}$ and let $c$ denote the non-trivial element of
$\Gal(F/F^{+})$. Suppose that the extension $F/F^{+}$ is unramified at
all finite places. Assume that $n[F^{+}:\bb{Q}]$ is divisible by
4. Under this assumption, we can find a reductive algebraic group $G$
over $F^{+}$ with the following properties:
\begin{itemize}
\item $G$ is an outer form of $\GL_n$ with $G_{/F} \cong \GL_{n/F}$;
\item for every finite place $v$ of $F^{+}$, $G$ is quasi-split at
  $v$;
\item for every infinite place $v$ of $F^{+}$, $G(F^{+}_v) \cong
  U_{n}(\bb{R})$.
\end{itemize}
We can and do fix a model for $G$ over the ring of integers
$\mc{O}_{F^{+}}$ of $F^{+}$ as in section 2.1 of \cite{ger}.
For each place $v$ of $F^{+}$ which splits as $ww^{c}$ in $F$ there is
a natural isomorphism
\[ \iota_{w} : G(F^{+}_v) \isoto \GL_n(F_w) \]
which restricts to an isomorphism between $G(\mc{O}_{F^{+}_v})$ and $\GL_n(\mc{O}_{F_w})$.
If $v$ is a place of $F^{+}$ split over $F$ and $\wt{v}$ is a place of
$F$ dividing $v$, then we let
\begin{itemize}
\item $\Iw(\wt{v})$ denote the subgroup of $\GL_n(\mc{O}_{F_{\wt{v}}})$
  consisting of matrices which reduce to an upper triangular matrix
  modulo $\wt{v}$.
\item $\Iw(\wt{v}^{b,c})$, for $0 \leq b \leq c$, denote the subgroup
  of $\GL_n(\mc{O}_{F_{\wt{v}}})$ consisting of matrices which reduce
  to an upper triangular matrix modulo $\wt{v}^{c}$ and to a unipotent
  matrix modulo $\wt{v}^{b}$. In particular $\Iw(\wt{v}^{0,0}) =
  \GL_n(\mc{O}_{F_{\wt{v}}})$. 
\end{itemize}

Let $l>n$ be prime number with the property that every place of
$F^{+}$ dividing $l$ splits in $F$. Fix an algebraic closure $\Qlbar$ of
$\bb{Q}_l$. Let $K$ be an algebraic extension of $\bb{Q}_l$ in
$\Qlbar$ such that every embedding $F \hookrightarrow \Qlbar$ has
image contained in $K$ and such that $K$ contains a primitive $l$-th
root of unity. Let $\mc{O}$ denote the ring of integers in $K$ and $k$ the residue field. Let
$S_l$ denote the set of places of $F^{+}$ dividing $l$ and for each $v
\in S_l$, let $\wt{v}$ be a place of $F$ over $v$. Let $\tilde{S}_l$
be the set of $\wt{v}$ for $v\in S_l$.

Let $W$ be an $\mc{O}$-module with an action of $G(\mc{O}_{F^{+},l})$. 
Let $V \subset G(\bb{A}_{F^+}^{\infty})$ be a compact open subgroup
with $v_l \in G(\mc{O}_{F^+,l})$ for all $v \in V$, where $v_l$
denotes the projection of $v$ to $G(F^{+}_l)$. We let $S(V,W)$ denote
the space of $l$-adic automorphic forms on $G$ of weight $W$ and level $V$, that is, the space of functions
\[  f : G(F^{+})\backslash G(\bb{A}_{F^+}^{\infty}) \rightarrow W \]
with $f(gv) = v_l^{-1} f(g)$ for all $v \in V$. 

Let $\wt{I}_{l}$ denote the set of embeddings $F \hookrightarrow K$
giving rise to a place in $\tilde{S}_l$. 
To each $\lambda \in (\bb{Z}^{n}_{+})^{\wt{I}_l}$ we
associate a finite free $\mc{O}$-module $M_{\lambda}$ with a
continuous action of $G(\mc{O}_{F^{+},l})$ as in Definition 2.2.3 of
\cite{ger}. The representation $M_{\lambda}$ is the tensor product
over $\tau \in \wt{I}_l$ of the irreducible algebraic representations
of $\GL_n$ of highest weights given by the $\lambda_\tau$. We write
$S_{\lambda}(V,\mc{O})$ instead of $S(V,M_{\lambda})$ and similarly
for any $\mc{O}$-module $A$, we write $S_{\lambda}(V,A)$ for
$S(V,M_{\lambda}\otimes_{\mc{O}}A)$.

Assume from now on that $K$ is a finite extension of $\Ql$. Let
$\mf{l}$ denote the product of all places in $S_l$. Let $R$ and $S_a$
denote finite sets of finite places of $F^{+}$ disjoint from each
other and from $S_l$ and consisting only of places which split in
$F$. Assume that each $v \in S_a$ is unramified over a rational prime $p$ with
$[F(\zeta_p):F]>n$. Let $T=S_l \coprod R \coprod S_a$. For each $v \in
T$ fix a place $\wt{v}$ of $F$ dividing $v$, extending the choice of
$\wt{v}$ for $v \in S_l$. We henceforth identify $G(F^{+}_{v})$ with
$GL_n(F_{\wt{v}})$ via $\iota_{\wt{v}}$ for $v \in T$ without comment.
Let $U=\prod_v U_v$ be the compact open subgroup of
$G(\bb{A}_{F^{+}}^{\infty})$ with
\begin{itemize}
\item $U_v = G(\mc{O}_{F^{+}_v})$ if $v \not \in R \cup S_a$ splits in
  $F$;
\item $U_v = \Iw(\wt{v})$ if $v \in R$;
\item $U_v = \ker( \GL_n(\mc{O}_{F_{\wt{v}}})
  \rightarrow \GL_n(k(\wt{v})))$ if $v \in S_a$;
\item $U_v$ is a hyperspecial maximal compact subgroup of $G(F^{+}_v)$
  if $v$ is inert in $F$.
\end{itemize}
If $0 \leq b \leq c$, we let $U(\mf{l}^{b,c})=U^{l} \times \prod_{v
  \in S_l} \Iw(\wt{v}^{b,c})$. We note that if $S_a$ is non-empty then $U$ is sufficiently small (which means that its projection to $G(F^{+}_v)$ for some place $v \in F^{+}$ contains no finite order elements other than the identity).

For each $v \in S_l$ fix a uniformizer $\varpi_{\wt{v}}$ in
$\mc{O}_{F_{\wt{v}}}$. For $0\leq b \leq c$ with $c>0$ and $j=1,\ldots,n$, consider the scaled Hecke
operator
\[
\Uop := \left(\prod_{i=1}^{j} \prod_{ \tau : F_{\wt{v}} \hookrightarrow \Qlbar}
  \tau(\varpi_{\wt{v}})^{- \lambda_{\tau|_{F}, n-i+1}} \right) \left[
  \Iw(\wt{v}^{b,c}) \left( \begin{matrix} \varpi_{\wt{v}}1_j & 0 \cr
    0 & 1_{n-j} \end{matrix} \right) \Iw(\wt{v}^{b,c}) \right]
\]
acting on the space $S_{\lambda}(U(\mf{l}^{b,c}),\mc{O})$. We let
$S_{\lambda}^{\ord}(U(\mf{l}^{b,c}),\mc{O})$
denote the ordinary part of $S_{\lambda}(U(\mf{l}^{b,c}),\mc{O})$ as defined in section 2.4 of
\cite{ger} (noting that the space
$S_{\lambda}(U(\mf{l}^{b,c}),\mc{O})$ is denoted
$S_{\lambda,\{1\}}(U(\mf{l}^{b,c}),\mc{O})$ in \cite{ger}). This is the maximal submodule on which each of the operators
$\Uop$ acts invertibly. This space is preserved by the Hecke operators
\begin{itemize}
\item 
\[ T_{w}^{(j)}:=  \iota_{w}^{-1}\left( \left[ GL_n(\mc{O}_{F_w}) \left( \begin{matrix}
      \varpi_{w}1_j & 0 \cr 0 & 1_{n-j} \end{matrix} \right)
GL_n(\mc{O}_{F_w}) \right] \right)
\]
for $w$ a place of $F$, split over $F^+$ and not lying over
$T$, $j=1,\ldots,n$ and
$\varpi_{w}$ a uniformizer in $\mc{O}_{F_w}$, and 
\item \[ \langle u \rangle := \prod_{v \in S_l} \left[
     \Iw(\wt{v}^{b,c})  \diag(u_{\wt{v}})
 \Iw(\wt{v}^{b,c}) \right] \]
for $u=  (u_{\wt{v}})_{v \in S_l} \in \prod_{v \in S_l}
(\mc{O}_{F_{\wt{v}}}^{\times})^{n}$.
\end{itemize}
We let $\bb{T}_{\lambda}^{T,\ord}(U(\mf{l}^{b,c}),\mc{O})$
denote the $\mc{O}$-subalgebra of $\End_{\mc{O}}(S_{\lambda}^{\ord}(U(\mf{l}^{b,c}),\mc{O}))$ generated by the operators
 $T_{w}^{(j)}$, $(T_{w}^{(n)})^{-1}$ and $\langle u \rangle$. We let
\[ \bb{T}_{\lambda}^{T,\ord}(U(\mf{l}^{\infty}),\mc{O}) :=
\varprojlim_{c} \bb{T}_{\lambda}^{T,\ord}(U(\mf{l}^{c,c}),\mc{O})
.\]
Let $T_n$ denote the diagonal torus in
$GL_n$. We define $T_n(\mf{l})$ as the pro-$l$ part of
$T_{n}(\mc{O}_{F^{+},l})=\prod_{v \in S_l}T_{n}(\mc{O}_{F^{+}_{v}})$. In other
words, we have an exact sequence
\[ 0 \rightarrow T_{n}(\mf{l}) \rightarrow T_{n}(\mc{O}_{F^{+},l})\rightarrow
T_{n}(\mc{O}_{F^{+}}/\mf{l}) \rightarrow 0 .\]
Define the completed group algebras
\begin{eqnarray*}
  \Lambda^{+} & := & \mc{O}[[T_{n}(\mc{O}_{F^{+},l})]] \\
  \Lambda & := & \mc{O}[[ T_{n}(\mf{l}) ]].
\end{eqnarray*}
Identifying $T_{n}(\mc{O}_{F^{+},l})$ with $\prod_{v \in S_l}
T_{n}(\mc{O}_{F_{\wt{v}}})$ in the natural way gives  $
\bb{T}_{\lambda}^{T,\ord}(U(\mf{l}^{\infty}),\mc{O})$
the structure of a $\Lambda^{+}$-algebra (via the operators $\langle u \rangle$).

It is shown in section 2.6 of \cite{ger} that the algebra
$\bb{T}_{\lambda}^{T,\ord}(U(\mf{l}^{\infty}),\mc{O})$
is independent of the weight $\lambda$
 in the sense that for each $\lambda$ there is a natural isomorphism $\bb{T}_{\lambda}^{T,\ord}(U(\mf{l}^{\infty}),\mc{O})\cong \bb{T}_{0}^{T,\ord}(U(\mf{l}^{\infty}),\mc{O})$.
We let
$\bb{T}^{T,\ord}(U(\mf{l}^{\infty}),\mc{O})$ denote
the universal ordinary Hecke algebra as in Definition 2.6.2 of \cite{ger}. By definition, this is just $\bb{T}_{0}^{T,\ord}(U(\mf{l}^{\infty}),\mc{O})$ with a modified
$\Lambda^{+}$-structure which is more convenient from the point of view of Galois representations.

\subsection{An $R^{\red}=\bb{T}$ Theorem}\label{sec:R=T}

Let $\TUniv$ be the algebra introduced above. Let $\mf{m}$ be a
maximal ideal of $\TUniv$ with residue field $k$ which is
non-Eisenstein in the sense of section 2.7 of \cite{ger}. According to
propositions 2.7.3 and 2.7.4 of \cite{ger} one can choose a continuous
homomorphism
\[ \rbar_{\mf{m}} : G_{F^{+}} \rightarrow
\mc{G}_n(\TUniv/\mf{m}) \]
with $\rbar_{\mf{m}}|_{G_{F}}$ absolutely irreducible and a continuous lifting
\[ r_{\mf{m}} : G_{F^{+}} \rightarrow \mc{G}_{n}(\TUniv_{\mf{m}}) \]
with the following properties:
\begin{enumerate}
\item[(0)] $r_{\mf{m}}^{-1}(\GL_n\times \GL_1)(\TUniv_{\mf{m}}) = G_{F}$.
\item $r_{\mf{m}}$ is unramified outside $T$.
\item If $v \not \in T$ is a place of $F^{+}$ which splits as $ww^{c}$
  in $F$ and $\Frob_w$ is the geometric Frobenius element of $G_{F_w}/I_{F_w}$, then
$r_{\mf{m}}(\Frob_{w})$ has characteristic polynomial
\[ X^{n}- T_{w}^{(1)} X^{n-1}+\ldots+ (-1)^{j} (\mathbf{N}w)^{j(j-1)/2} T_{w}^{(j)}X^{n-j} + \ldots + (-1)^{n}
(\mathbf{N}w)^{n(n-1)/2} T_{w}^{(n)} .\]

\item $\nu \circ r_{\mf{m}} = \epsilon^{1-n} \delta_{F/F^{+}}^{\mu_\mf{m}}$ where $\delta_{F/F^+}$ is the
non-trivial character of $\Gal(F/F^{+})$ and $\mu_{\mf{m}} \in \bb{Z}/2\bb{Z}$.

\item If $v \in R$ and $\sigma \in I_{F_{\wt{v}}}$, then $r_{\mf{m}}(\sigma)$ has characteristic polynomial $(X-1)^{n}$.
\end{enumerate}
We make the following \emph{assumptions}:
\begin{itemize}
\item[(a)] The subgroup $\overline{r}_{\mf{m}}(G_{F^{+}(\zeta_l)})$ of
  $\mc{G}_{n}(k)$ is big;
\item[(b)] For $v \in S_l \cup R$,
  $\overline{r}_{\mf{m}}(G_{F_{\wt{v}}})=\{ 1_n\}$;
\item[(c)] The set $S_a$ is non-empty and for $v \in S_a$, $\overline{r}_{\mf{m}}|_{G_{F_{\wt{v}}}}$ is
  unramified and $H^{0}(G_{F_{\wt{v}}},\ad
  \overline{r}_{\mf{m}}(1))=\{ 0\}$.
\end{itemize}

For $v \in S_l$, let
\begin{eqnarray*}
  \Lambda_{F_{\wt{v}}} & := & \mc{O}[[I_{F_{\wt{v}}^{\ab}}(l)^{n}]] 
  \end{eqnarray*}
where for a group $H$, $H(l)$ denotes the pro-$l$ completion.  The
inverses of the Artin maps $\Art_{F_{\wt{v}}}$ for $v \in S_l$ give
rise to an isomorphism
\[\prod_{v \in S_l} (I_{F_{\wt{v}}^{\ab}}(l))^{n}    \tilde{\longrightarrow}  \prod_{v \in S_l} (1+ \varpi_{\wt{v}}\mc{O}_{F_{\wt{v}}})^{n} \cong   T_{n}(\mf{l}) \]
and hence an isomorphism
\[ \widehat{\otimes}_{v \in S_l} \Lambda_{F_{\wt{v}}} \tilde{\longrightarrow}
\Lambda .\]

Corollary 3.1.4 of \cite{ger} shows that $r_{\mf{m}}$ satisfies the following property, in addition to (0)-(4) above:
\begin{itemize}
\item[(5)] For $v \in S_l$, the homomorphism
  $R^{\square}_{\rbarwtv}\widehat{\otimes}_{\mc{O}}\Lambda_{F_{\wt{v}}}
  \rightarrow \TUniv_{\mf{m}}$ coming from
  $r_{\mf{m}}|_{G_{F_{\wt{v}}}}$ and the
  $\Lambda_{F_{\wt{v}}}$-algebra structure on $\TUniv_{\mf{m}}$
  factors through the quotient
  $R^{\triangle}_{\rbarwtv,\Lambda_{F_{\wt{v}}}}$ of
  $R^{\square}_{\rbarwtv}\widehat{\otimes}_{\mc{O}}\Lambda_{F_{\wt{v}}}$
  constructed in section \ref{sec: p=l case in non-fixed weight}.
\end{itemize}

We now turn to deformation rings.
For each $v \in S_l$, let
$R^{\triangle}_{\rbarwtv,\Lambda_{F_{\wt{v}}}}$ be the quotient of
$R^{\square}_{\rhobar}\widehat{\otimes}_{\mc{O}}\Lambda_{F_{\wt{v}}}$
constructed in section \ref{sec: p=l case in non-fixed weight}.  For
$v$ in $R$, let $R_{\rmbarwtv}^{1}$ denote the quotient of
$R_{\rmbarwtv}^{\square}$ corresponding to lifts $r$ for which
$r(\sigma)$ has characteristic polynomial $(X-1)^n$ for each $\sigma
\in I_{F_{\wt{v}}}$.  This ring is studied in section 3 of
\cite{tay06}.
Let 
\[ \mc{S}_{\Lambda} =
\left(F/F^+,T,\wt{T},\Lambda,\rbar,\epsilon^{1-n}\delta_{F/F^+}^{\mu_{\mf{m}}},\{
  R_{\rmbarwtv}^{\square}\}_{v \in S_a},\{ R_{\rmbarwtv}^{1}\}_{v \in
    R}, \{ R_{\rmbarwtv,\wt\Lambda_{\wt{v}}}^{\triangle}\}_{v \in
    S_l}\right) \] Let $\mc{C}_{\Lambda}$ denote the category of
complete local Noetherian $\Lambda$-algebras with residue field
$k$. We say that a lift $r$ of $\rbar$ to an object $A$ of
$\mc{C}_{\Lambda}$ is of type $\mc{S}_{\Lambda}$ if for each $v \in
S_l$, the homomorphism
$R^{\square}_{\rbarwtv}\widehat{\otimes}_{\mc{O}}\Lambda_{F_{\wt{v}}}
\rightarrow A$ coming from $r|_{G_{F_{\wt{v}}}}$ and the
$\Lambda$-structure on $A$ factors through
$R^{\triangle}_{\rbarwtv,\Lambda_{F_{\wt{v}}}}$ and if for each $v \in
R$ the homomorphism $R^{\square}_{\rbarwtv} \rightarrow A$ coming from
$r|_{G_{F_{\wt{v}}}}$ factors through $ R_{\rmbarwtv}^{1}$. We define
deformations of type $\mc{S}_{\Lambda}$ in the same way.  Let
$\Def_{\mc{S}_{\Lambda}} : \mc{C}_{\Lambda} \rightarrow Sets$ be the
functor which sends an object $A$ to the set of deformations of
$\rbar$ to $A$ of type $\mc{S}_{\Lambda}$. This functor is represented
by an object $R^{\univ}_{\mc{S}_{\Lambda}}$ of $\mc{C}_{\Lambda}$.

Properties (0)-(5) above imply that the lift $r_{\mf{m}}$ of $\rbar_{\mf{m}}$ to  $\TUniv_{\mf{m}}$ is of type $\mc{S}_{\Lambda}$ and hence gives rise to a homomorphism of $\Lambda$-algebras
\[ R_{\mc{S}_{\Lambda}}^{\univ} \rightarrow \TUniv_{\mf{m}}. \]
The following result is contained in Theorem 4.3.1 of \cite{ger}.

\begin{thm}
\label{thm: R=T}
  The map $R_{\mc{S}_{\Lambda}}^{\univ} \rightarrow \TUniv_{\mf{m}}$ induces an isomorphism
  \[ (R_{\mc{S}_{\Lambda}}^{\univ})^{\red} \isoto \TUniv_{\mf{m}} \]
and $\mu_{\mf{m}}\equiv n \mod 2$ so that $\rbar_{\mf{m}}$ is odd.
\end{thm}

Let $\lambda \in (\bb{Z}^{n}_{+})^{\wt{I}_l}$ and for each $v \in
S_l$, let $\lambda_{\wt{v}}$ denote the element of
$(\bb{Z}^n_+)^{\Hom(F_{\wt{v}},K)}$ given by the $\lambda_{\tau|_{F}}$
for $\tau : F_{\wt{v}}\hookrightarrow K$. In section \ref{subsubsec:
  the case where p=l} we associated to $\lambda_{\wt{v}}$ an $n$-tuple
of characters
$(\chi^{\lambda_{\wt{v}}}_1,\ldots,\chi^{\lambda_{\wt{v}}}_n)$ from
$I_{F_{\wt{v}}}$ to $\mc{O}^{\times}$. These characters induce an
$\mc{O}$-algebra homomorphism
\[ \chi^{\lambda_{\wt{v}}}: \Lambda_{F_{\wt{v}}} \rightarrow \mc{O} 
\]
and taking the tensor product over the places $v \in S_l$, we get a homomorphism
\[ \chi^{\lambda} : \Lambda \rightarrow \mc{O}. \]
We denote the kernels of these homomorphisms by $\wp_{\lambda_{\wt{v}}}$ and $\wp_{\lambda}$. The next result follows from
Corollary 2.5.4 and Lemma 2.6.4  of \cite{ger} (noting that $U$ is sufficiently small since $S_a$ is non-empty).

\begin{prop}
\label{prop: T is finite and faithful over Lambda}
  The algebra $\TUniv$ is finite and faithful as a $\Lambda$-module
  and for every $\lambda \in (\bb{Z}^{n}_{+})^{\wt{I}_l}$ there is a
  natural surjection
\[  \bb{T}^{T,\ord}(U(\mf{l}^{\infty}),\mc{O})\otimes_{\Lambda}\Lambda_{\wp_{\lambda}}/\wp_{\lambda} \onto \bb{T}^{T,\ord}_{\lambda}(U(\mf{l}^{1,1}),\mc{O})\otimes_{\mc{O}}K \] 
whose kernel is nilpotent.
\end{prop}

Let $\lambda$ and $\lambda_{\wt{v}}$ for $v \in S_l$ be as
above. Consider the deformation problem
\[
\mc{S}_{\lambda}=(F/F^{+},T,\wt{T},\mc{O},\rbar,\epsilon^{1-n}\delta_{F/F^+}^{\mu_{\mf{m}}},\{
R_{\rmbarwtv}^{\square}\}_{v \in S_a},\{ R_{\rmbarwtv}^{1}\}_{v \in
  R}, \{ R_{\rmbarwtv}^{\triangle_{\lambda_{\wt{v}}},cr}\}_{v \in
  S_l}) \]
and the corresponding deformation ring
$R^{\univ}_{\mc{S}_{\lambda}}$. 

\begin{cor}
\label{cor : R is finite over O}
  The ring $R^{\univ}_{\mc{S}_{\lambda}}$ is a finite $\mc{O}$-algebra.
\end{cor}

\begin{proof}
First of all, observe that if $R$ is an object of $\mc{C}_{\mc{O}}$,
then $R$ is a finite $\mc{O}$-algebra if $R^{\red}$ is. Indeed, if
$R^{\red}$ is finite over $\mc{O}$ then $R/\mf{m}_{\mc{O}}$ is Noetherian and
0-dimensional and hence Artinian. It follows from the topological
Nakayama lemma that $R$ is finite over $\mc{O}$.

The ring $(R^{\univ}_{\mc{S}_{\lambda}})^{\red}$ is a quotient of
$(R^{\univ}_{\mc{S}_{\Lambda}})^{\red}/\wp_{\lambda}$. By Theorem \ref{thm:
  R=T} and Proposition \ref{prop: T is finite and faithful over
  Lambda}, $(R^{\univ}_{\mc{S}_{\Lambda}})^{\red}/\wp_{\lambda}$ is
a finite $\mc{O}$-algebra. The result follows.
\end{proof}

\section{Existence of lifts}\label{sec:existence of lifts}\subsection{}

Let $F$ be an imaginary CM field, $F^+$ its maximal totally real subfield and $c$
the non-trivial element of $\Gal(F/F^+)$. Let $\pi$ be a RACSDC automorphic
representation of $\GL_n(\bb{A}_F)$ and $\iota$ an isomorphism $\Qlbar
\isoto \bb{C}$. In \cite{chenevierharris} it is shown that there is a semisimple representation
\[ r_{l,\iota}(\pi) : G_F \rightarrow \GL_n(\Qlbar) \]
uniquely determined by the following properties:
\begin{enumerate}
\item $r_{l,\iota}(\pi)^c = r_{l,\iota}(\pi)^{\vee}\epsilon^{1-n}$;
\item for $w$ a place of $F$ not dividing $l$ we have
\[ r_{l,\iota}(\pi)|_{G_{F_w}}^{\sesi} \cong
r_{l}(\iota^{-1}\pi_w)^{\vee}(1-n)^{\sesi} \]
where $r_{l}(\iota^{-1}\pi_w)$ is the $l$-adic represenation
associated to the Weil-Deligne representation
$\rec_l(\pi_w^{\vee}\otimes |\;\;|^{(1-n)/2})$ (and $\rec_l$ is the
local Langlands correspondence of \cite{MR1876802}).
\end{enumerate}

If $F$ and $c$ are as above, we let
$(\bb{Z}^n_{+})^{\Hom(F,\Qlbar)}_c$ denote the subset of
$(\bb{Z}^{n}_+)^{\Hom(F,\Qlbar)}$ consisting of elements $\lambda$
with  $\lambda_{\tau c,j}=-\lambda_{\tau,n-j+1}$ for all
  $\tau:F\into \Qlbar$ and $j=1,\ldots,n$.
If $\lambda \in (\bb{Z}^n_{+})^{\Hom(F,\Qlbar)}_c$ and $w$ is a place
of $F$ dividing $l$, we let
$\lambda_{w}=(\lambda_{w,\sigma})_{\sigma}$ be the element of
$(\bb{Z}^n_+)^{\Hom(F_{w},\Qlbar)}$ determined by
$\lambda_{w,\sigma}=\lambda_{\sigma|_{F}}$ for all $\sigma :
F_{w}\into \Qlbar$.

One can find a finite extension $K$ of $\bb{Q}_l$ with ring of
integers $\mc{O}$ so that $r_{l,\iota}(\pi)$ can be conjugated to take
values in $\GL_n(\mc{O})$. Reducing modulo the maximal ideal of
$\mc{O}$, extending scalars to $\Flbar$ and semisimplifying, one
obtains a representation $\rbar_{l,\iota}(\pi):G_{F} \rightarrow \GL_n(\Flbar)$ which is independent
of any choices made.

If $K$ (resp.\ $k$) is an algebraic extension of $\bb{Q}_l$ (resp.\ $\bb{F}_l$) and $\rho : G_{F}
\rightarrow \GL_n(K)$ (resp.\  $\rhobar : G_{F}
\rightarrow \GL_n(k)$)  is a continuous representation, we say that
$\rho$ (resp.\ $\rhobar$) is automorphic if there exists a $\pi$ and $\iota$ as above with
$r_{l,\iota}(\pi)$ (resp.\ $\rbar_{l,\iota}(\pi)$) isomorphic to $\rho
\otimes_{K}\Qlbar$ (resp.\ $\rhobar\otimes_k \overline{\bb{F}}_l$). We
say that $\rho$ (or $\rhobar$) is ordinarily automorphic if in addition $\pi$ and
$\iota$ can be chosen so that $\pi$ is $\iota$-ordinary at every place dividing $l$. We say that $\rho$ is ordinary
automorphic of weight  $\lambda\in(\bb{Z}^n_{+})^{\Hom(F,\Qlbar)}_c$
if $\rho$ is automorphic and $\rho|_{G_{F_w}}$ is ordinary of weight $\lambda_w \in (\bb{Z}^n_+)^{\Hom(F_w,\Qlbar)}$ for
each place $w|l$ of $F$. We say that $\rho$ is ordinary automorphic if
it is ordinary automorphic of some weight.
If $\rho$ is ordinarily automorphic and its
  reduction $\rhobar$ is absolutely irreducible, then $\rho$ is
  ordinary automorphic by Proposition 5.3.1 of \cite{ger}.

We are now ready to prove our main theorem. For the convenience of the
reader, we recall all our assumptions in the statement of the
theorem.

\begin{thm}\label{thm: main result for unitary groups} Let $F$ be an imaginary CM field with maximal totally real
  subfield $F^+$. Let $n\geq 2$ be an integer and $l>n$ a prime
  number.  Suppose that $\zeta_l \not \in F$. Assume that the extension $F/F^+$ is split at all places
  dividing $l$.  Suppose that \[\rhobar:G_{F}\to\GL_n(\Flbar)\] is an
  irreducible representation satisfying the following assumptions.
  \begin{enumerate}
  \item The representation $\rhobar$ is ordinarily automorphic (so in
    particular $\rhobar^c\cong\rhobar^\vee\overline{\epsilon}^{1-n}$).
  \item Any place of $F$ at which $\rhobar$ is ramified splits over
    $F^+$.
  \item The image $\rhobar(G_{F(\zeta_l)})$ is big.
  \item $\overline{F}^{\ker\ad \rhobar}$ does not contain $F(\zeta_l)$.
  \item There is an element $\lambda\in(\Z^n_{+})^{\Hom(F,\Qlbar)}_c$
    such that for every place $w|l$ of $F$, $\rhobar|_{G_{F_w}}$ has
    an ordinary crystalline lift of weight $\lambda_w$. 
  \end{enumerate}
  Then $\rhobar$ has an automorphic lift $\rho$ which is crystalline
  and ordinary of weight $\lambda_w$ at each place $w$ of $F$ dividing
  $l$, and which is ordinarily automorphic of level prime to $l$.

In fact, suppose we are given a set of places $S$ of $F^+$ which split in
$F$, a choice of a place $\tilde{v}$ of $F$ above each place $v$ of
$F^+$, and an inertial type $\tau_{\tilde{v}}$ for $I_{F_{\tilde{v}}}$ for each $v\in S$ not dividing $l$ such that
$\rhobar|_{G_{F_{\tilde{v}}}}$ has a lift of type $\tau_{\tilde{v}}$. Then $\rho$ can be
chosen to be of type $\tau_{\tilde{v}}$ at $\tilde{v}$ for all places $v\in S$,
$v\nmid l$. More precisely, given a choice of irreducible component of each ring
$R_{\rhobar|_{G_{F_{\tilde{v}}}}}^{\sq,\tau_{\tilde{v}}}$ with $v\in S$, $v\nmid
l$ and of each ring  $R_{\rhobar|_{G_{F_{\tilde{v}}}}}^{\triangle_{\lambda_{\wt{v}}},cr}$
with $v|l$, $\rho$ may be chosen so as to
give a point on each of these components.
\end{thm}
\begin{proof} It suffices to prove the final statement. Choose an
  isomorphism $\iota: \Qlbar \isoto \bb{C}$ and a
  RACSDC automorphic representation $\pi$ of $\GL_n(\bb{A}_F)$ which is $\iota$-ordinary at all places dividing $l$  such that
  $\rbar_{l,\iota}(\pi)\cong \rhobar$.  By Lemma 2.1.4 of \cite{cht},
  $\rhobar$ extends to a representation
  $\rbar:G_{F^+}\to\mathcal{G}_n(\Flbar)$ with
  $\rbar^{-1}(\GL_n(\Flbar)\times \GL_1(\Flbar) )= G_{F}$. 
Extending $S$ if necessary, we may assume that $S$ contains all places above $l$ and that $\rhobar$ is unramified away from $S$. Indeed, for the places $v \nmid l$ just added to $S$, the lift $r_{l,\iota}(\pi)|_{G_{F_{\wt{v}}}}$ determines an inertial type $\tau_{\wt{v}}$ for $I_{F_{\wt{v}}}$ and at least one irreducible component of $R^{\square,\tau_{\wt{v}}}_{\rhobar|_{G_{F_{\tilde{v}}}}}$.

 Choose a finite extension $K$ of $\Ql$ inside $\Qlbar$ with residue
 field $k$ and ring of integers $\mc{O}$ containing a primitive $l$-th
 root of unity so that $r_{l,\iota}(\pi)$
 can be conjugated to take values in $\mc{G}_n(\mc{O})$, so that $K$
 contains the image of every embedding $F \into \Qlbar$ and so that
 each type $\tau_{\tilde{v}}$ for $v \in S, v \nmid l$ is defined over
 $K$. Assume from now on that $r_{l,\iota}(\pi)$ takes values in $\mc{G}_n(\mc{O})$.

By Lemma 2.1.4 of \cite{cht},  $\nu \circ \rbar =
\overline{\epsilon}^{1-n}\delta_{F/F^+}^{\mu}$ for some $\mu \in
\bb{Z}/2\bb{Z}$, where  $\delta_{F/F^+}$ is the quadratic character of
$G_{F^+}$ corresponding to $F$. By Theorem \ref{thm: R=T}, $\rbar$ is
odd, so in fact $\mu \equiv n \mod 2$.
For each $v \in S$, let $R_{\wt{v}}$ be the chosen irreducible component of 
 $R^{\square,\tau_{\tilde{v}}}_{\rhobar|_{G_{F_{\wt{v}}}}}$ when $v\nmid l$ or $R^{\triangle_{\lambda_{\wt{v}}},cr}_{\rhobar|_{G_{F_{\wt{v}}}}}$ when $v|l$. Let $\wt{S}$ denote the set of $\wt{v}$ for $v \in S$ and let 
\[ \mc{S} = (F/F^+,S,\wt{S},\mc{O},\rbar,\epsilon^{1-n}\delta_{F/F^+}^{\mu},\{ R_{\wt{v}}\}_{v \in S}). \]
To prove the theorem it suffices to show that we can find a closed point of $R^{\univ}_{\mc{S}}[1/l]$ so that the corresponding representation restricted to $G_F$ is automorphic.

Choose a finite place $v_1$ of $F$ not lying over $S$ so that
\begin{itemize}
\item $v_1$ is unramified over a rational prime $p$ with $[F(\zeta_p):F]>n$;
\item $v_1$ does not split completely in $F(\zeta_l)$;
\item $\ad \rbar(\Frob_{v_1}) = 1$.
\end{itemize}
The last two conditions imply that $H^0(G_{F_{v_1}},\ad \rbar(1))=\{0\}$.
Choose a CM extension $L$ of $F$ with the following
properties:
\begin{itemize}
\item $L/F$ is Galois and soluble;
\item $L$ is linearly disjoint from $\overline{F}^{\ker(\ad \rbar)}(\zeta_l)$ over $F$;
\item all primes of $L$ lying above $S$ or $\{v_1\}$ are split over
  $L^+$ where $L^+$ is the maximal totally real subfield of
  $L$;
\item the extension $L/L^+$ is unramified at all finite places;
\item $4|[L^+:F^+]$;
\item for each place $\wt{v} \in \wt{S}$ away from $l$ and each place $w$ of $L$ lying over $\wt{v}$, we have $\mathbf{N}w\equiv 1 \mod l$, $\rhobar(G_{L_{w}})=\{ 1_n \}$, the type $\tau_{\wt{v}}$ becomes trivial upon restriction to $I_{L_w}$ and if $\pi_L$ denote the base change of $\pi$ to $L$, then $(\pi_L)_w^{\Iw(w)} \neq 0$.
\item the places $\{v_1,cv_1\}$ split completely in $L$;
\item for each place $\wt{v} \in \wt{S}$ dividing $l$ and each place $w$ of $L$ lying over $\wt{v}$ we have $\rhobar(G_{L_w})=\{ 1_n\}$.
\item if $w$ is a place of $L$ not lying over $l$ such that
  $(\pi_L)_w$ is ramified, then $w$ lies over a place of $L^+$ which
  splits in $L$, and $(\pi_L)_w^{\Iw(w)} \neq 0$.
\end{itemize}
Let $T$ denote the set of places of $L^+$ comprised of those lying
above $S\cup \{v_{1}|_{F^+}\}$, together with any places of $L^+$ over
which there is a place $w$ of $L$ with $(\pi_L)_w$ ramified. Let
$\wt{T}$ denote a set of places of $L$, containing all places
lying above $\wt{S}\cup \{v_1\}$, such that $\wt{T}$ consists of one
place $\wt{w}$ for each place $w \in T$. For each $\wt{w} \in \wt{T}$
lying above $v_1$, let
$R_{\wt{w}}=R^{\square}_{\rhobar|_{G_{L_{\wt{w}}}}}$. For $\wt{w} \in
\wt{T}$ not dividing $l$ or $v_1$, let
$R_{\wt{w}}$ denote the quotient $R^{1}_{\rhobar|_{G_{L_{\wt{w}}}}}$
of $R^{\square}_{\rhobar|_{G_{L_{\wt{w}}}}}$ corresponding to lifts
for which each element of inertia has characteristic polynomial
$(X-1)^n$. Let $\lambda_{L}$ be the element of
$(\bb{Z}^n_{+})^{\Hom(L,\Qlbar)}_c$ determined by
$(\lambda_L)_{\tau}=\lambda_{\tau|_F}$ for all $\tau : L \into
\Qlbar$. Extend $K$ if necessary so that it contains the image of
every embedding $L \into \Qlbar$. For $\wt{w} \in \wt{T}$ lying above
$l$, let $R_{\wt{w}} =
R^{\triangle_{(\lambda_{L})_{\wt{w}}},cr}_{\rhobar|_{G_{L_{\wt{w}}}}}$. Let
\[ \mc{S}^{\prime} =
(L/L^+,T,\wt{T},\mc{O},\rbar|_{G_{L^+,T}},\epsilon^{1-n}\delta_{L/L^+}^{\mu},\{
R_{\wt{w}}\}_{w \in T}) .\]
Restricting the universal deformation over $R^{\univ}_{\mc{S}}$ to $G_{L^{+},T}$ gives rise to a map $R^{\univ}_{\mc{S}^{\prime}} \rightarrow R^{\univ}_{\mc{S}}$ and by
 Lemma \ref{lem: deformation ring is finite over another for the
   unitary group}, this map is finite. 

 Now, let $G$ be a reductive group over $\mc{O}_{L^+}$ as in section
 \ref{subsec:auto forms on unitary groups} (with $L^+$ replacing
 $F^+$). By Th\'eor\`eme 5.4 and Corollaire 5.3 of \cite{labesse} and
 the assumption that $L/L^+$ is unramified at all finite places,
 $\pi_L$ is the strong base change of an automorphic representation
 $\Pi$ of $G(\bb{A}_{L^+})$. By Lemma 5.1.6 of \cite{ger} $\pi_L$ is
 $\iota$-ordinary at each place of $L$ dividing
 $l$. Let $U\subset G(\bb{A}_{L^+}^{\infty})$ be the compact open
 subgroup defined as in section \ref{subsec:auto forms on unitary
   groups} with $S_a$ the set of places of $T$ above $v_1|_{F^+}$ and
 $R$ the set of places of $T$ not dividing $l$ and not in $S_a$. Then
 extending $\mc{O}$ if necessary, the Hecke eigenvalues on
 $(\iota^{-1}\Pi^{l,\infty})^{U^l}\bigotimes \otimes_{v|l}(\iota^{-1}
 \Pi_{v})^{\ord}$ give rise to an $\mc{O}$-algebra homomorphism
 $\TUniv \rightarrow \mc{O}$. Reducing this modulo the maximal ideal
 of $\mc{O}$ gives a maximal ideal $\mf{m}$ of $\TUniv$ which is
 non-Eisenstein by the second of our conditions on $L$ above. All of
 the hypotheses of section \ref{sec:R=T} are satisfied and we deduce
 from Corollary \ref{cor : R is finite over O} that
 $R^{\univ}_{\mc{S}^{\prime}}$ is finite over $\mc{O}$. Theorem \ref{thm: R=T} and Proposition \ref{prop: T is finite and
   faithful over Lambda} imply that every closed point of
 $R^{\univ}_{\mc{S}^{\prime}}[1/l]$ gives rise to a representation of
 $G_{L}$ which is ordinarily automorphic.

 Since $R^{\univ}_{\mc{S}}$ is finite over $\mc{O}$ and has Krull
 dimension at least one by Lemma \ref{lem: lower bound on unitary
   dimension}, the ring $R^{\univ}_{\mc{S}}[1/l]$ is non-zero. Any
 closed point on this ring gives rise to a crystalline ordinary representation $\rho$ of
 $G_{F}$ which is ordinarily automorphic upon restriction to $G_L$.
 By Lemma 1.4 of \cite{BLGHT} any such $\rho$ is automorphic and hence, by
 Lemma 5.1.6 of \cite{ger}, is in fact ordinarily
 automorphic. Finally, it follows from Theorem 5.3.2 of \cite{ger}
 that such a $\rho$ is ordinarily automorphic of level prime $l$.
\end{proof}
We can frequently make this rather more explicit.

\begin{cor}
  \label{cor: explicit main unitary theorem in regular
    weight} Let $F$ be an imaginary CM field with maximal totally real
  subfield $F^+$. Let $n\geq 2$ be an integer and $l>n$ a prime
  number.  Suppose that $\zeta_l \not \in F$. Assume that the extension $F/F^+$ is split at all places
  dividing $l$. Let $\tilde{S}_l$ be a set of places of $F$ lying over
  $l$, containing exactly one place above each place of $F^+$ dividing
  $l$. Suppose that \[\rhobar:G_{F}\to\GL_n(\Flbar)\] is an
  irreducible representation satisfying the following assumptions.
  \begin{enumerate}
 \item The representation $\rhobar$ is ordinary automorphic (so in
    particular $\rhobar^c=\rhobar^\vee\epsilon^{1-n}$).
 \item Any place of $F^+$ at which $\rhobar$ is ramified splits in $F$.
  \item The image $\rhobar(G_{F(\zeta_l)})$ is big.
  \item $\overline{F}^{\ker\ad \rhobar}$ does not contain $F(\zeta_l)$.
  \item There is an element $\lambda\in(\Z^n_{+})^{\Hom(F,\Qlbar)}_c$ such that
        for every place $v\in\tilde{S}_l$, $\rhobar|_{G_{F_v}}$ is
    isomorphic to a representation \[ \left(\begin{matrix} \overline{\mu}_{v,1} & * & \ldots & * &* \cr 0 &
   \overline{\mu}_{v,2} & \ldots & * &* \cr
       \vdots & \vdots &\ddots & \vdots& \vdots \cr 0 & 0 & \ldots &
    \overline{\mu}_{v,n-1} &*\cr 0 & 0 & \ldots & 0&
     \overline{\mu}_{v,n}
     \end{matrix} \right)
   \] where
 $\overline{\mu}_{v,i}|_{I_{F_v}}=\overline{\chi}_{i}^{\lambda_v}|_{I_{F_v}}$
 (where $\chi_{i}^{\lambda_v}$ is the crystalline character of
 Definition \ref{defn: chars associated to lambda}), and for each $i<j$ we have
  $\overline{\mu}_{v,i}\overline{\mu}_{v,j}^{-1}\neq \epsilon$.
 
  \end{enumerate}
  Then $\rhobar$ has an ordinarily automorphic lift (of level prime to
  l) $\rho$ of weight
  $\lambda$ which is crystalline at all places dividing $l$; furthermore for each place
  $v\in\tilde{S}_l$, $\rho|_{G_{F_v}}$ is isomorphic to a representation \[ \left(\begin{matrix} \psi_{v,1} & * & \ldots & * &* \cr 0 &
      \psi_{v,2} & \ldots & * &* \cr
       \vdots & \vdots &\ddots & \vdots& \vdots \cr 0 & 0 & \ldots &
      \psi_{v,n-1} &*\cr 0 & 0 & \ldots & 0&
      \psi_{v,n}
     \end{matrix} \right)
  \]with $\psi_{v,i}$ a lift of $\overline{\mu}_{v,i}$ agreeing with
  $\chi_{i}^{\lambda_v}$ on $I_{F_v}$.
\end{cor}
\begin{proof}
  This is an immediate consequence of Theorem \ref{thm: main result
    for unitary groups}, Lemma \ref{lem:Ordinary always lifts in
    regular weight} and Lemma \ref{lem: picking components of the ordinary crystalline
    deformation ring}.
\end{proof}

We now state a corollary to the proof of Theorem \ref{thm: main result for
  unitary groups}, which we will use in section \ref{sec:gsp4}.

\begin{cor}\label{cor: universal unitary ring has o-rank at least
   one}
Let $l,n,F,F^{+},\rhobar,\pi,\rbar,S,\wt{S},\{\tau_{\wt{v}}\}_{v \in
  S, v \nmid l}, \{ \lambda_{\wt{v}} \}_{v \in S_l}$ and
$\{ R_{\wt{v}} \}_{v \in S}$ be as in Theorem
\ref{thm: main result for unitary groups} and its proof. Let
\[ \mc{S} =
(F/F^+,S,\wt{S},\mc{O},\rbar,\epsilon^{1-n}\delta_{F/F^+}^{n},\{
R_{\wt{v}} \}_{v\in S}).\]
Then $R^{\univ}_{\mc{S}}$ is a finite $\mc{O}$-module of rank at
least 1.
\end{cor}

\section{Serre weights}\label{sec:serre weights}\subsection{}We now
put ourselves in the setting of section \ref{subsec:auto forms on
  unitary groups}. In particular, we let $F^{+}$ denote a totally real
number field and $n$ a positive integer. Let $F/F^{+}$ be a totally
imaginary quadratic extension of $F^{+}$ and let $c$ denote the
non-trivial element of $\Gal(F/F^{+})$. Suppose that the extension $F/F^{+}$ is
unramified at all finite places. Assume that $n[F^{+}:\bb{Q}]$ is
divisible by 4. We note that we could make somewhat weaker
assumptions, but the necessity of considering definite unitary groups
which fail to be quasi-split at some finite places would complicate
the notation unnecessarily.

Let $G$ be the reductive algebraic group over $F^+$ defined in section
\ref{subsec:auto forms on unitary groups}, together with a fixed model
over $\bigO_{F^+}$ as before. Again, we take a prime number $l>n$ so
that every place in the set $S_l$ of places of $F^+$ dividing $l$
splits in $F$. Fix a set $\tilde{S}_l$ of places of $F$ consisting of
exactly one place above each place in $S_l$. Let $\bigO$ be the ring
of integers of $\Qlbar$, with residue field $\Flbar$. Let
$\tilde{I}_l$ denote the set of embeddings $F\into \Qlbar$ giving rise
to one of the places $\tilde{v}\in\tilde{S}_l$. Let
$\tilde{I}_{\tilde{v}}$ denote the subset of $\tilde{I}_l$ giving rise
to $\tilde{v}$. Let the residue field of $F_{\tilde{v}}$ be
$k(\tilde{v})$. Then any element $\sigma\in\tilde{I}_{\tilde{v}}$
induces an embedding $\overline{\sigma}:k(\tilde{v})\into \Flbar$. For
an embedding $\tau:k(\tilde{v})\into \Flbar$, we let $\tilde{I}_\tau$
denote the subset of $\tilde{I}_{\tilde{v}}$ consisting of the
$\sigma$ with $\overline{\sigma}=\tau$. We let $\overline{I}_l$ be the
set of embeddings $\tau:k(\tilde{v})\into \Flbar$ (running over all
$v$).

Define $(\Z^n_+)^{\tilde{I}_l}$ as in section \ref{subsec:auto forms
  on unitary groups}. Let $(\Z^n_+)^{\overline{I}_l}$ be the subset of
$(\Z^n)^{\overline{I}_l}$ consisting of $\lambda$ with
$l-1\ge\lambda_{\tau,i}-\lambda_{\tau,i+1}\ge 0$ for all $\tau$ and
all $i=1,\dots,n-1$. Let $(\Z^n_{+,res})^{\tilde{I}_l}$ denote the
subset of $(\Z^n_+)^{\tilde{I}_l}$ consisting of weights $\lambda$ with the
property that for each $\tilde{v}$ and $\tau:k(\tilde{v})\into \Flbar$, it is
possible to write
$\tilde{I}_\tau=\{\sigma_{\tau,1},\dots,\sigma_{\tau,e}\}$ with
$\lambda_{\sigma_{\tau,i},j}=0$ if $i>1$ and $l-1\ge\lambda_{\sigma_{\tau,1},j}-\lambda_{\sigma_{\tau,1},j+1}\ge 0$ for all $j=1,\dots,n-1$. This being the case, we define
$\lambda_{\tau,j}:=\lambda_{\sigma_{\tau,1},j}$. In this way, we
define a surjective map $\pi$
from  $(\Z^n_{+,res})^{\tilde{I}_l}$ to $(\Z^n_+)^{\overline{I}_l}$.

Fix $\lambda\in (\Z^n_{+,res})^{\tilde{I}_l}$. We now consider the finite free $\bigO$-module $M_\lambda$ of
Definition 2.2.3 of \cite{ger}. This has a continuous action of
$G(\bigO_{F^+,l})=\prod_{v\in S_l}\GL_n(\bigO_{F_{\tilde{v}}})$. The
action on $M_\lambda\otimes \Flbar$ factors through $\prod_{v\in S_l}\GL_n(k(\tilde{v}))$.

Let $S_a$ be a nonempty set of finite places of $F^+$, disjoint from $S_l$, such that any
place $v$ of $S_a$ splits in $F$ and is unramified over a rational prime $p$
with $[F(\zeta_p):F]>n$. Choose a place $\tv$ of $F$ lying over each
$v \in S_a$. Let $U=\prod_vU_v$ be a compact open subgroup
of $G(\A_{F^+}^\infty)$ such that
\begin{itemize}
\item $U_v\subset G(\bigO_{F^+_v})$ for all $v$;
  \item $U_v=\iota_{\tv}^{-1}\ker(\GL_n(\bigO_{F_{\tilde{v}}})\to\GL_n(k(\tilde{v})))$
    if $v\in S_a$;
  \item $U_v=G(\bigO_{F^+_v})$ if $v|l$;
  \item $U_v$ is a hyperspecial maximal compact subgroup of $G(F_v^+)$
    if $v$ is inert in $F$.
\end{itemize}
Then (because $S_a$ is nonempty) $U$ is sufficiently small,
and \[S_\lambda(U,\bigO)\otimes \Flbar=S(U,M_\lambda)\otimes
\Flbar=S(U,M_\lambda\otimes \Flbar).\]Let $T$ be a finite set of finite places
of $F^+$ which split in $F$, containing $S_l$ and all the places $v$ which split
in $F$ for which $U_v\neq G(\bigO_{F^+_v})$. We let
$\mathbb{T}_\lambda^{T,univ}$ be the commutative $\bigO$-polynomial algebra
generated by formal variables $T_w^{(j)}$ for all $1\le j\le n$, $w$ a place
of $F$ lying over a place $v$ of $F^+$ which splits in $F$ and is not
contained in $T$, together with variables $T_{\lambda,\tilde{v}}^{(j)}$ for
all $v\in S_l$ and $j=1,\ldots,n$. We now fix a uniformiser $\varpi_{\tilde{v}}$ of
$\bigO_{F_\tv}$ for each $v\in S_l$. The algebra
$\mathbb{T}_\lambda^{T,univ}$ acts on $S(U,M_\lambda)$ via the
following Hecke operators:
\begin{itemize}
\item 
\[ T_{w}^{(j)}:=  \iota_{w}^{-1} \left[ GL_n(\mc{O}_{F_w}) \left( \begin{matrix}
      \varpi_{w}1_j & 0 \cr 0 & 1_{n-j} \end{matrix} \right)
GL_n(\mc{O}_{F_w}) \right] 
\] for $w\not \in T$ and $\varpi_w$ a uniformiser in $\mc{O}_{F_w}$, and 
\item \[T_{\lambda,\tv}^{(j)}:= \left(\prod_{i=1}^{j} \prod_{ \tau : F_{\wt{v}} \hookrightarrow \Qlbar}
  \tau(\varpi_{\wt{v}})^{- \lambda_{\tau|_{F}, n-i+1}} \right) \iota_{\tv}^{-1}\left[  GL_n(\mc{O}_{F_\tv}) \left( \begin{matrix} \varpi_{\wt{v}}1_j & 0 \cr
    0 & 1_{n-j} \end{matrix} \right) GL_n(\mc{O}_{F_\tv}) \right].
\] for $v\in S_l$.
\end{itemize}

Suppose that $\mf{m}$ is a maximal ideal of
$\mathbb{T}_\lambda^{T,univ}$ with residue field $\Flbar$ such that
$S(U,M_\lambda)_{\mf{m}}\neq 0$. Then as in section \ref{sec:R=T}
there is a continuous semisimple
representation \[\rbar_{\mf{m}}:G_{F^+}\to\Gn(\Flbar)\]naturally associated
to $\mf{m}$. We have the following definition.
\begin{defn}
  Suppose that $\rhobar:G_F\to\GL_n(\Flbar)$ is a continuous
  irreducible representation. Then we say that $\rhobar$ is modular
  and ordinary of weight $\lambda\in  (\Z^n_{+,res})^{\tilde{I}_l}$ if
  there is a $U$, $T$ as
  above and a maximal ideal $\mathfrak{m}$ of
  $\mathbb{T}^{T,univ}_\lambda$ with residue field $\Flbar$ such
  that
  \begin{itemize}
  \item $S(U,M_\lambda)_{\mf{m}}\neq 0$,
  \item $\mathfrak{m}$ does not contain any of the
    operators $T^{(j)}_{\lambda,\tv}$, and
  \item $\rhobar\cong \rbar_{\mathfrak{m}}|_{G_{F}}$.
  \end{itemize}
We say that $\rhobar$ is modular and ordinary if it is modular and
ordinary of some weight $\lambda\in  (\Z^n_{+,res})^{\tilde{I}_l}$.
\end{defn}
Note in particular that if $\rhobar$ is ordinary and modular then it
is unramified at any place of $F$ which doesn't split over
$F^+$. We have the following theorem.

\begin{thm}\label{thm:serre wts; char 0 version}
   Suppose that \[\rhobar:G_{F}\to\GL_n(\Flbar)\] is an 
  irreducible representation satisfying the following assumptions.
  \begin{enumerate}
  \item The representation $\rhobar$ is modular and ordinary (so in
    particular $\rhobar^c=\rhobar^\vee\epsilon^{1-n}$).
  \item The image $\rhobar(G_{F(\zeta_l)})$ is big.
  \item $\overline{F}^{\ker\ad \rhobar}$ does not contain $F(\zeta_l)$.
   \end{enumerate}
  Then $\rhobar$ is modular and ordinary of weight $\lambda\in
  (\Z^n_{+,res})^{\tilde{I}_l}$ if and only if
  \begin{itemize}
  \item For every place $v|l$ of $F^+$, $\rhobar|_{G_{F_{\tilde{v}}}}$ has
    an ordinary crystalline lift of weight $\lambda_{\tilde{v}}$. 
  \end{itemize}

\end{thm}
\begin{proof}
  Suppose firstly that  $\rhobar$ is modular and ordinary of weight $\lambda\in
  (\Z^n_{+,res})^{\tilde{I}_l}$. Then by definition we see that there is
  a $U$, $T$, $\mathfrak{m}$ as above such that
  $S_\lambda(U,\bigO)_{\mathfrak{m}}\neq 0$ and
  $\rhobar_{\mathfrak{m}}\cong\rhobar$. Then
  $S_\lambda(U,\Qlbar)_{\mathfrak{m}}\neq 0$. Choose an isomorphism
  $\iota:\Qlbar \isoto \bb{C}$. We see by Corollaire 5.3 of
  \cite{labesse} and Lemma 2.2.5 of \cite{ger}
  that there is a RACSDC representation $\pi$ of $\GL_n(\A_F)$ which
  is unramified at $l$, $\iota$-ordinary at all $v|l$ (by Lemma 2.7.6 of \cite{ger}) and which satisfies
  $\rbar_{l,\iota}(\pi)\cong\rhobar$. Thus $r_{l,\iota}(\pi)$ is ordinary and
  crystalline of weight $\lambda$. The representations
  $r_{l,\iota}(\pi)|_{G_{F_{\tilde{v}}}}$ then provide the required lifts.

  For the converse, if the condition holds then by Theorem \ref{thm:
    main result for unitary groups} $\rhobar$ has a lift to a
  representation $\rho$ which is crystalline and ordinary of weight
  $\lambda$ and ordinarily automorphic of level prime to $l$.
 Say $\rho=r_{l,\iota}(\pi)$. The
  result now follows from Corollaire 5.3 and Th\'{e}or\`{e}me 5.4 of
  \cite{labesse}, the strong multiplicity one theorem for $\GL_n$ and 
Lemma 2.2.5 of \cite{ger}.  
\end{proof}

We now show that if $\rhobar$ is modular and ordinary of weight
$\lambda$, then it is modular of weight $\pi(\lambda)$ in the sense of
generalisations of Serre's conjecture (cf. \cite{herzigthesis}). This
is a straightforward consequence of the elementary calculations
underlying Hida theory, as we now explain.

Let $v_\lambda$ be the rank one $\bigO$-submodule of $M_\lambda$ on
which the usual maximal torus of $\GL_n$ acts via the highest weight
$\lambda$. Let $v_{w_0\lambda}$ be the rank one $\bigO$-submodule of $M_\lambda$ on
which the usual maximal torus of $\GL_n$ acts via the lowest weight
$w_0\lambda$.

The irreducible $\Flbar$-representations of $\prod_{v\in
  S_l}\GL_n(k(\tilde{v}))$ are tensor products of irreducible
representations of the $\GL_n(k(\tilde{v}))$. From the standard
classification of the irreducible $\Flbar$-representations of
$\GL_n(k(\tilde{v}))$ (see for example the appendix to
\cite{herzigthesis}), one sees that:
\begin{enumerate}
\item There is an irreducible $\Flbar$-representation $F_\lambda$ of $\prod_{v\in
  S_l}\GL_n(k(\tilde{v}))$ for each
  $\lambda\in(\Z^n_+)^{\overline{I}_l}$, and every irreducible
  $\Flbar$-representation  of $\prod_{v\in
  S_l}\GL_n(k(\tilde{v}))$ is equivalent to some $F_\lambda$.
\item Take $\lambda\in(\Z^n_{+,res})^{\tilde{I}_l}$. Let $P_\lambda$ be the sub-$\prod_{v\in
  S_l}\GL_n(k(\tilde{v}))$-representation of
$M_\lambda\otimes \Flbar$ generated by $v_\lambda\otimes \Flbar$. Then
$P_\lambda\cong F_{\pi(\lambda)}$ (see II.8.8(1) of \cite{MR2015057}).
\item $P_\lambda$ contains $v_{w_0\lambda}\otimes \Flbar$.
\end{enumerate}

We have the following straightforward but crucial lemma.

\begin{lemma}\label{lem:ordinary component is the socle}
  Let $T_l$ be the Hecke operator $\prod_{v|l,1\le j\le
    n}T_{\lambda,\tv}^{(j)}$. Then $T_l$ acts by $0$ on
  $S(U,(M_\lambda\otimes \Flbar)/P_\lambda)$.
\end{lemma}
\begin{proof} We employ the notation of
  section 2 of \cite{ger}; specifically, let $\alpha$ be the element
  of $G(\A_{F^+}^\infty)$ defined in the proof of Lemma 2.5.2 of
  \cite{ger}. Then $(w_0\lambda)(\alpha)^{-1}\alpha$ acts by an
  $l$-adic unit on $v_{w_0\lambda}$, but by an eigenvalue with
  positive $l$-adic valuation on every other weight space in
  $M_\lambda$. The result follows immediately from the fact that $P_\lambda$
  contains $v_{w_0\lambda}\otimes \Flbar$.
 \end{proof}
Note that if $U$, $T$, $\lambda$ are as above and $\mf{m}$ is a maximal ideal of
  $\mathbb{T}^{T,univ}_\lambda$ with residue field $\Flbar$ such
  that $S(U,P_\lambda)_{\mathfrak{m}}\neq 0$, then
  $S(U,M_\lambda)_{\mf{m}}\neq 0$, and we have a Galois representation
  $\rbar_{\mf{m}}$ as before.
 \begin{cor}
   $\rhobar$ is modular and ordinary of weight $\lambda$ if and only
   if there is a $U$, $T$ as
  above and a maximal ideal $\mathfrak{m}$ of
  $\mathbb{T}^{T,univ}_\lambda$ with residue field $\Flbar$ such
  that
  \begin{itemize}
  \item $S(U,P_\lambda)_{\mathfrak{m}}\neq 0$,
  \item $\mathfrak{m}$ does not contain any of the
    operators $T^{(j)}_{\lambda,\tv}$, and
  \item $\rhobar\cong \rbar_{\mathfrak{m}}|_{G_{F}}$.

  \end{itemize}
 \end{cor}
 \begin{proof}
   This is an immediate consequence of the definitions and of Lemma \ref{lem:ordinary component is the socle}.
 \end{proof}

 Fix now an element $\mu\in(\Z^n_{+})^{\overline{I}_l}$. Fix
 $\lambda\in(\Z^n_{+,res})^{\tilde{I}_l}$ with
 $\pi(\lambda)=\mu$. Then there is an equivalence $P_\lambda\cong
 F_\mu$, so that $\bb{T}^{T,univ}_\lambda$ acts on $S(U,F_\mu)$. Suppose
 that $\lambda'\in(\Z^n_{+,res})^{\tilde{I}_l}$ with
 $\pi(\lambda')=\mu$. Then we also have an action of
 $\bb{T}^{T,univ}_{\lambda'}$ on $S(U,F_\mu)$, and it is easy to check that
 under the natural isomorphism between $\bb{T}^{T,univ}_{\lambda'}$ and
 $\bb{T}^{T,univ}_\lambda$ the Hecke operators at places not dividing $l$
 act identically on $S(U,F_\mu)$, while those at places dividing $l$
 differ only by units (this is just a statement about the rescaling in
 the definition of the Hecke operators at places dividing $l$). We
 thus have the following result/definition.

 \begin{lemma}\label{lem: defining ordinary modular in char l}
   Suppose that $\rhobar:G_F\to\GL_n(\Flbar)$ is a continuous
  irreducible representation. Then we say that $\rhobar$ is modular
  and ordinary of weight  $\mu\in(\Z^n_{+})^{\overline{I}_l}$ if
  there is a $U$, $T$ as
  above, and for some (equivalently, any) $\lambda\in(\Z^n_{+,res})^{\tilde{I}_l}$ with
 $\pi(\lambda)=\mu$ there is a maximal ideal $\mathfrak{m}$ of
  $\mathbb{T}^{T,univ}_\lambda$ with residue field $\Flbar$ such
  that
  \begin{itemize}
  \item $S(U,F_\mu)_{\mf{m}}\neq 0$,
  \item $\mathfrak{m}$ does not contain any of the
    operators $T^{(j)}_{\lambda,\tv}$, and
  \item $\rhobar\cong \rbar_{\mathfrak{m}}|_{G_{F}}$.
  \end{itemize}
\end{lemma}

We can then reinterpret Theorem \ref{thm:serre wts; char 0 version}.

\begin{thm}\label{thm:serre wts; char l version}
  Suppose that \[\rhobar:G_{F}\to\GL_n(\Flbar)\] is an 
  irreducible representation satisfying the following assumptions.
  \begin{enumerate}
  \item The representation $\rhobar$ is ordinary and modular (so in
    particular $\rhobar^c=\rhobar^\vee\epsilon^{1-n}$).
  \item The image $\rhobar(G_{F(\zeta_l)})$ is big.
  \item $\overline{F}^{\ker\ad \rhobar}$ does not contain $F(\zeta_l)$.
   \end{enumerate}
  Then $\rhobar$ is modular and ordinary of weight
  $\mu\in(\Z^n_{+})^{\overline{I}_l}$ if and only if for some (equivalently, any) $\lambda\in(\Z^n_{+,res})^{\tilde{I}_l}$ with
 $\pi(\lambda)=\mu$, the following condition holds. 
  \begin{itemize}
  \item For every place $v|l$ of $F^+$, $\rhobar|_{G_{F_{\tilde{v}}}}$ has
    an ordinary crystalline lift of weight $\lambda_{\tilde{v}}$. 
  \end{itemize}

\end{thm}
\begin{proof}
  This follows at once from Theorem \ref{thm:serre wts; char 0
    version}, Lemma \ref{lem:ordinary component is the socle}, and
  Lemma \ref{lem: defining ordinary modular in char l}.
\end{proof}

\section{$\GSp_4$}\label{sec:gsp4}

\subsection{Definitions}We define $\GSp_4$ to be the reductive group
over $\Z$
defined as a subgroup of $\GL_4$
by \[\GSp_4(R)=\{g\in\GL_4(R):gJ{}^tg=\mu(g)J\}\]
where $\mu(g)$ is the
similitude factor (which is uniquely determined by $g$), and $J$ is
the antisymmetric matrix \[\begin{pmatrix}0& X\\-X&0\end{pmatrix}\]where $X$ is the
$2\times 2$ antidiagonal matrix with all entries on the antidiagonal
equal to $1$. Note that the map $\mu:g\mapsto\mu(g)$ gives a
homomorphism $\GSp_4\to\Gm$.

\begin{lem}\label{lem: symplectic representation valued over trace}Let
  $\Gamma$ be a profinite group, and $S\subset R$ be complete local
  Noetherian rings with $\m_R\cap S=\m_S$ and common residue
  field $k$ of characteristic $>2$. Let $\rho:\Gamma\to\GSp_4(R)$ be a continuous representation. Suppose that $\rho\text{ mod }\m_R$ is absolutely irreducible
  and that $\tr\rho(\Gamma)\subset S$. Then there is a
  $\ker(\GSp_4(R)\to\GSp_4(k))$-conjugate of $\rho$ whose image is
  contained in $\GSp_4(S)$.
  
\end{lem}
\begin{proof} By Lemma 2.1.10 of \cite{cht}, we see that
  $\rho$ is $\ker(\GL_4(R)\to\GL_4(k))$-conjugate to a
  representation $\rho'$ valued in $\GL_4(S)$. Now, $(\mu\circ\rho)^2=\det\rho=\det\rho'$ is valued in $S$, which
  by Hensel's lemma means that $\mu\circ\rho$ is valued in $S$. Thus
  ${}^t(\rho')^{-1}(\mu\circ\rho)$ is also valued in
  $\GL_4(S)$. Because $\rho'$ and  ${}^t(\rho')^{-1}(\mu\circ\rho)$
  are conjugate in $\GL_4(R)$ they are also conjugate in $\GL_4(S)$,
  by Th\'{e}or\`{e}me 1 of \cite{MR1279611}. Suppose that $\rho'=B{}^t(\rho')^{-1}(\mu\circ\rho)B^{-1}$. The matrix $B$ is
  antisymmetric (because $\rho$ is symplectic).  By choosing a symplectic basis for the symplectic form
  determined by $B$, we see that $\rho$ is $\GL_4(R)$-conjugate to a
  representation valued in $\GSp_4(S)$, and it is easy to check that
  one may choose the symplectic basis so that the conjugating matrix
  is in $\ker(\GL_4(R)\to\GL_4(k))$. It remains to check that
the conjugating matrix is also in $\GSp_4(R)$; but this is an immediate consequence of
  Schur's lemma.
\end{proof}

\subsection{Symplectic lifting rings (local case)}

Fix as before a finite field $k$
of characteristic $l>2$, and a finite totally ramified extension $K$ of $W(k)[1/l]$ with
ring of integers $\bigO$. Let the maximal ideal of $\bigO$ be $\mf{m}_K=(\pi_K)$.
Let $M$ be a finite extension of $\bb{Q}_p$ for some prime $p$,
possibly equal to $l$. In the case where $p=l$, we assume that $K$ contains the image of
every embedding of $M$ into $\overline{K}$. Let 
\[\rhobar: G_M \to \GSp_4(k)\] 
be a continuous representation. Since $\GSp_4$ is an algebraic
subgroup of $\GL_4$, we can also view it as a representation to
$\GL_4(k)$. Then there is a universal
$\mc{O}$-lifting 
\[\rho^{\sq}:G_M \to\GL_4(R^{\sq}_{\rhobar}),\]
and it is immediate that there is a quotient $R^{\sq,sympl}_{\rhobar}$ of
$R^{\sq}_{\rhobar}$ and a universal symplectic
lifting 
\[\rho^{\sq,sympl}:G_M \to\GSp_4(R^{\sq,sympl}_{\rhobar}).\]

We will need to study
certain refined lifting problems. Suppose that $p=l$. Let $\lambda$ be an element of $(\bb{Z}^{4}_+)^{\Hom(M,K)}$ and
let $\mathbf{v}_{\lambda}$ be the associated $l$-adic Hodge type (see
section \ref{subsubsec: the case where p=l}). 
Corollary 2.7.7 of \cite{kisinpst} shows that there is a unique
$l$-torsion-free quotient
$R^{sympl,\mathbf{v}_{\lambda},cr}_{\rhobar}$ of
$R^{\sq,sympl}_{\rhobar}$ with the property that for any finite
$K$-algebra $B$, a homomorphism of $\bigO$-algebras
$R^{\sq}_{\rhobar}\to B$ factors through
$R^{sympl,\mathbf{v}_{\lambda},cr}_{\rhobar}$ if and only if the
corresponding representation is crystalline of $l$-adic Hodge type
$\mathbf{v}_{\lambda}$ (where as usual we define a homomorphism
$G_M\to\GSp_4(B)$ to be crystalline of a particular Hodge type if and
only if the same is true of the composite homomorphism to $\GL_4(B)$).

The following discussion will be useful below. Let $E$ be a finite
extension of $K$ and let $\mc{C}_E$ be the category of finite, local
$E$-algebras with residue field $E$. If $B$ is an object of
$\mc{C}_E$, a \emph{symplectic $B$-module} is a pair $(V_B,\alpha_B)$ 
where $V_B$ is a free $B$-module of rank 4 with a continuous
action of $G_M$ and $\alpha_B$ is a symplectic pairing $ V_B \times V_B
\rightarrow B$ satisfying
\[ \alpha_B(gx,gy)= \psi_B(g)\alpha_B(x,y) \] for all $x,y \in V_B$
and $g \in G_M$, for some continuous character $\psi_B : G_M
\rightarrow B^\times$.  A \emph{symplectic basis} of such a pair
$(V_B,\alpha_B)$ is a basis $\beta_B=\{ e_1,e_2,e_3,e_4\}$ of $V_B$
where the matrix $(\alpha_B(e_i,e_j))$ equals $\lambda J$ for some
$\lambda \in B^\times$. Two symplectic $B$-modules $(V_B,\alpha_B)$ and $(V_B^{\prime},\alpha_B^{\prime})$ are isomorphic if $\psi_B = \psi_B^{\prime}$ and there is an isomorphism of $B[G_M]$-modules $V_B \cong V_B^{\prime}$ under which $\alpha_B^{\prime}$ pulls back to $\alpha_B$.

Fix a symplectic $E$-module $(V_E,\alpha_E)$ together with a
symplectic basis $\beta_E$.  A \emph{deformation} of $(V_E,\alpha_E)$
to an object $B$ of $\mc{C}_E$ is a triple $(V_B,\alpha_B,\iota_B)$
where $(V_B,\alpha_B)$ is a symplectic $B$-module and $\iota_B$ is an
isomorphism $(V_B \otimes_B B/\mf{m}_B, \alpha_B \otimes_{B}
B/\mf{m}_B)\cong (V_E,\alpha_E)$ of symplectic $E$-modules. A
\emph{framed deformation} of $(V_E,\alpha_E,\beta_E)$ is a deformation
$(V_B,\alpha_B,\iota_B)$ together with a symplectic basis $\beta_B$ of
$(V_B,\alpha_B)$ reducing to $\beta_E$ under $\iota_B$. Let $\rho_E :
G_M \rightarrow \GSp_4(E)$ be the matrix of $V_E$ with respect to
$\beta_E$. For an object $B$ of $\mc{C}_E$ there is a natural
bijection between the set of isomorphism classes of framed
deformations of $(V_E,\alpha_E,\beta_E)$ to $B$ and the set of lifts
$\rho_B : G_M \rightarrow \GSp_4(B)$: the class of a framed
deformation $(V_B,\alpha_B,\beta_B)$ corresponds to the matrix
representation of $V_B$ with respect to the basis $\beta_B$.
Similarly, there is a natural bijection between the set of isomorphism
classes of deformations of $(V_B,\alpha_B)$ to $B$ and the set of
\emph{deformations of $\rho_E$} to $B$, that is,
$\ker(\GSp_4(B)\rightarrow \GSp_4(E))$-conjugacy classes of lifts
$\rho_B : G_M \rightarrow \GSp_4(B)$ of $\rho_E$: the class of a
deformation $(V_B,\alpha_B)$ corresponds to the conjugacy class of the
matrix representation of $V_B$ with respect to any symplectic basis
$\beta_B$ lifting $\beta_E$.

Suppose that $(V_B,\alpha_B)$ is a \emph{crystalline} symplectic
$B$-module and let $D_{B}:=D_{\cris}(V_B)=(V_B \otimes_{\bb{Q}_l}
B_{\cris})^{G_M}$ be the associated weakly admissible filtered
$\varphi$-module. Let $D_{\psi_B}=D_{\cris}(\psi_B)$.
There is an associated alternating pairing
\[D(\alpha_B): D_{B} \times D_{B} \rightarrow D_{\psi_B} \]
which is a map of filtered $\varphi$-modules and is non-degenerate in
the sense that it induces an isomorphism $D_{B} \rightarrow \Hom(D_{B},D_{\psi_B})$.
This pairing is defined by taking the $B_{\cris}$-linear extension of $\alpha_B$
 to
$V_{B}\otimes_{\bb{Q}_l} B_{\cris}$ and then taking
$G_M$-invariants. 
Suppose in addition that $V_B$ has $l$-adic Hodge type
$\mathbf{v}_{\lambda}$. Let $\tau : M \into K$ be an embedding and let
$D_{B,\tau}=(D_{B}\otimes_{M_0}M)\otimes_{B\otimes M,1\otimes \tau}B$
and $D_{\psi_B,\tau}=(D_{\psi_B}\otimes_{M_0}M)\otimes_{B\otimes
  M,1\otimes \tau}B$. Then $D(\alpha_B)$ induces a symplectic pairing
$D_{B,\tau}\times D_{B,\tau} \rightarrow D_{\psi_B,\tau}$.
For $j=1,\ldots,4$, let $i_j =
\lambda_{\tau,j}+(4-j)$ be the Hodge-Tate weights of $B$ with respect
to $\tau$. Let $i_{\psi}$ be the Hodge-Tate weight of $\psi_B$ with
respect to $\tau$. Then $i_{\psi}=i_1+i_4=i_2+i_3$ since $V_{B} \cong \Hom_{B}(V_{B},\psi_B)$. Let
$\Fil^{i}$ be the filtration on $D_{B,\tau}$.
 In order for $D(\alpha_B)$ to respect filtrations
and to be non-degenerate we must have
$D(\alpha_B)(\Fil^{i_1},\Fil^{i_3})=\{0\},D(\alpha_B)(\Fil^{i_2},\Fil^{i_3})=D_{\psi_B,\tau}$
and $D(\alpha_B)(\Fil^{i_1},\Fil^{i_4})=D_{\psi_B,\tau}$. 
In other
words, we can find a symplectic basis $e_1,e_2,e_3,e_4$ for
$D_{B,\tau}$ such that $\Fil^{i_j}=Be_1 + \ldots + Be_j$ for
$j=1,\ldots,4$.

We define a \emph{symplectic filtered $\varphi$-module} over an
object $B$ in $\mc{C}_E$ to be a pair $(D_B,D(\alpha_B))$ consisting
of a weakly admissible rank 4 filtered
$\varphi$-module $D_B$ over $B\otimes_{\bb{Q}_l}M_0$ and an
alternating, non-degenerate morphism of filtered $\varphi$-modules
\[ D(\alpha_B) : D_B \times D_B \rightarrow D_{\psi_B} \]
where $D_{\psi_B}$ is a weakly admissible rank 1 filtered
$\varphi$-module over $B\otimes_{\bb{Q}_l}M_0$. There is an obvious notion of isomorphism between
symplectic filtered $\varphi$-modules and also an obvious notion of a
deformation of a symplectic filtered $\varphi$-module over $E$ to an
object $B$ of $\mc{C}_E$.
The functors $D_{\cris}$ and $V_{\cris}$ are quasi-inverse equivalences of categories between
the category of crystalline symplectic $B$-modules and the category
of symplectic filtered $\varphi$-modules over $B$ (all morphisms in
these categories are isomorphisms).

Suppose now that $M$ is a finite extension of $\Qp$, $p\neq l$. Then
it is easy to check (for example by considering the Weil-Deligne
representation corresponding to the universal lifting) that the
inertial type at a closed point of the generic fibre is an invariant
of the irreducible components of $R^{\sq,sympl}_{\rhobar}[1/l]$. Thus
for any $4$-dimensional inertial type $\tau$ of $I_M$ which is defined
over $K$, there is a unique reduced $l$-torsion-free quotient
$R^{sympl,\tau}_{\rhobar}$ of $R^{\sq,sympl}_{\rhobar}$, corresponding
to a union of irreducible components of
$R^{\sq,sympl}_{\rhobar}[1/l]$, with the property that for any finite
extension $L$ of $K$, a homomorphism of $\bigO$-algebras
$R^{\sq,sympl}_{\rhobar}\to L$ factors through
$R^{sympl,\tau}_{\rhobar}$ if and only if the corresponding lifting of
$\rhobar$ (considered as a representation to $\GL_4(L)$) has type
$\tau$.

In applications we will fix the similitude character of our
deformations. To this end, fix a character
$\psi:G_M \to\bigO^\times$ lifting the character $\mu\circ\rhobar$
which is crystalline if $p=l$.
Let $R^{\sq,sympl,\psi}_{\rhobar}$ denote the universal lifting ring for
lifts with similitude factor $\psi$. Similarly, we let
$R^{sympl,\tau,\psi}_{\rhobar}$,
$R^{sympl,\mathbf{v}_{\lambda},cr,\psi}_{\rhobar}$ denote the
corresponding quotients of $R^{\sq,sympl,\psi}_{\rhobar}$. 

Let $\ad\rhobar$ denote the Lie algebra of $\GSp_4$ over $k$, and
$\ad^0\rhobar$ the Lie algebra of $\Sp_4$. These
have a natural action of $G_{M}$ via $\rhobar$ and the adjoint
action of $\GSp_4(k)$, and are respectively $11$-dimensional and
$10$-dimensional $k$-vector spaces.

We have the following result on the dimensions of these local
lifting rings.

\begin{prop}\label{prop: dimension of local symplectic deformation rings}
  Let $M$ be a finite extension of $\Qp$. If $p\neq l$, and $\tau$ is
  such that the ring $R^{sympl,\tau,\psi}_{\rhobar}$ is non-zero, then any
  irreducible component of $R^{sympl,\tau,\psi}_{\rhobar}$ has dimension at
  least $11$. If $p=l$ 
    and $\mathbf{v}_{\lambda}$ is such that 
  $R^{sympl,\mathbf{v}_{\lambda},cr,\psi}_{\rhobar}$ is non-zero, then
  this ring is
  equidimensional of dimension $11+4[M:\Ql]$.
\end{prop}

\begin{proof}
  Firstly, suppose $p=l$ and let $X= \Spec
  R^{sympl,\mathbf{v}_{\lambda},cr,\psi}_{\rhobar}$. Let $x$ be a closed point of $X[1/l]$
  with residue field $E$. It suffices to show that the completed local
  ring $\mc{O}_{X,x}^{\wedge}$ is formally smooth over $E$ of
  dimension $10+4[M:\bb{Q}_l]$. We first establish formal
  smoothness. Let $\rho_{E} : G_M \rightarrow \GSp_4(E)$ be the
  representation associated to $x$.  Let $B$ denote a finite local
  $E$-algebra with residue field $E$ and let $I$ be an ideal of $B$
  with $\mf{m}_{B}I=\{0\}$. Let $\zeta :
  R^{sympl,\mathbf{v},cr,\psi}_{\rhobar} \rightarrow B/I$ be an $\mc{O}$-algebra
  homomorphism corresponding to a crystalline lift $\rho_{B/I} : G_M
  \rightarrow \GSp_4(B/I)$ of $\rho_E$.
  We need to show that we can lift $\zeta$ to $B$, or equivalently,
  that we can find a crystalline lift $G_M \rightarrow \GSp_4(B)$ of
  $\rho_{B/I}$ with similitude character $\psi$.

  Let $V_{B/I}=(B/I)^4$ regarded as $G_M$-module via $\rho_{B/I}$ and
  let $\alpha_{B/I} : V_{B/I} \times V_{B/I} \rightarrow (B/I)(\psi)$ be the
  symplectic pairing associated to the matrix $J$ (that is,
  $\alpha_{B/I}(x,y)={}^tx J y$ where $x$ and $y$ are regarded as column
  vectors).  
    Let $(D_{B/I},D(\alpha_{B/I}))$ be the symplectic, filtered
  $\varphi$-module over $B/I$ associated to $(V_{B/I},\alpha_{B/I})$.
  To construct the required lift of $\rho_{B/I}$, it suffices (by applying
  $V_{\cris}$) to construct a symplectic filtered
  $\varphi$-module $(D_B,D(\alpha_B))$ over $B$ (with $D(\alpha_B)$
  valued in $B \otimes_{E} D_{\cris}(E(\psi))$) lifting $(D_{B/I},D(\alpha_{B/I}))$.

  Let $b$ be an $E\otimes_{\bb{Q}_l}M_0$-generator of $D_{\psi}:=D_{\cris}(E(\psi))$.
  Choose a $(B/I) \otimes_{\bb{Q}_l} M_0$-basis $e_1,e_2,e_3,e_4$ for
  $D_{B/I}$ so that the matrix $(D(\alpha_{B/I})(e_i,e_j))$ is $(1\otimes b)J
  \in M_{4\times 4}((B/I)\otimes_{E}D_{\psi})$. The matrix
  $M_{\varphi}$ of $\varphi$ with respect to this basis is an element
  of $\GSp_4((B/I)\otimes_{\bb{Q}_l} M_0)$ with similitude factor
  `$\varphi(b)/b$'$\in (E\otimes_{\bb{Q}_l}M_0)^{\times}\subset ((B/I)\otimes_{\bb{Q}_l}M_0)^{\times}$.  Let
  $\wt{M}_{\varphi}$ be a lifting of this matrix to an element of
  $\GSp_4(B\otimes_{\bb{Q}_l}M_0)$ with the same similitude factor.
  Let $D_{B}$ be the free $B\otimes_{\bb{Q}_l} M_0$-module on
  generators $\wt{e}_1,\wt{e}_2,\wt{e}_3,\wt{e}_4$. Endow it with the
  symplectic form $D(\alpha_{B}): D_B \times D_B \to B\otimes_E D_{\psi}$ defined by
  $(D(\alpha_B)(\wt{e}_i,\wt{e}_j))=(1\otimes b)J$. Let $\wt\varphi$
  be the $\varphi_0$-semilinear automorphism of $D_{B}$ whose matrix
  with respect to the basis $\wt{e}_{i}$ is $\wt{M}_{\varphi}$. Now
  choose a filtration on $D_{B}\otimes_{M_0} M$ lifting the
  filtration on $D_{B/I} \otimes_{M_0}M$ and such that $D(\alpha_B)$
  becomes a map of filtered $\varphi$-modules.
  Then $D_{B}$ becomes a weakly admissible
  filtered $\varphi$-module and we have shown that $\mc{O}_{X,x}^{\wedge}$ is formally smooth over $E$.

  We now determine the relative dimension $d$ of $\mc{O}_{X,x}^{\wedge}$
  over $E$. Let $\mf{g}$ denote the Lie algebra of $\GSp_4(E)$ and
  $\mf{g}^\circ$ the Lie algebra of $\Sp_4(E)$.
 Let $D^{\square}_{\rho_E}(E[\varepsilon])$ (resp.\
 $D_{\rho_E}(E[\varepsilon])$) denote the set of crystalline lifts
 (resp.\ deformations) $G_M \rightarrow \GSp_4(E[\varepsilon])$ of
 $\rho_E$ with similitude character $\psi$. These sets are naturally
 $E$-vector spaces. Since the natural map
 $D_{\rho_E}^{\square}(E[\varepsilon])\onto
 D_{\rho_E}(E[\varepsilon])$ is a $\mf{g}^{\circ}/(\mf{g}^{\circ})^{G_M}$-torsor, we have
 \[ d = \dim_E D^{\square}_{\rho_E}(E[\varepsilon]) = \dim_E \left(
   \mf{g}^{\circ}/ (\mf{g}^{\circ})^{G_M} \right) + \dim_E
 D_{\rho_E}(E[\varepsilon]). \] 
 Let $D_{D_E}(E[\varepsilon])$ denote the set of equivalence classes
 of deformations $(D,D(\alpha))$ to $E[\varepsilon]$ of the symplectic filtered $\varphi$-module
 $(D_{E},D(\alpha_E))$ where the pairing $D(\alpha)$ takes values in
 $E[\varepsilon]\otimes_E D_{\psi}$. By the discussion
 preceding the proposition, we see that there is a natural bijection between
 $D_{\rho_E}(E[\varepsilon])$ and $D_{D_E}(E[\varepsilon])$.

 Choose any deformation $(D,D(\alpha))$ in
 $D_{D_E}(E[\varepsilon])$. Given any other such deformation
 $(D^{\prime},D(\alpha)^{\prime})$, we can choose an isomorphism of
 $E[\varepsilon]\otimes_{\bb{Q}_l}M_0$-modules $j:
 D^{\prime} \rightarrow D$ taking $D(\alpha)$ to
 $D(\alpha)^{\prime}$. Let $\varphi$ denote the $\varphi$-operator on
 $D$ and $\Fil$ the filtration on $D\otimes_{M_0} M$.
 Let $\varphi^{\prime}$ denote the operator on
 $D$ corresponding under $j$ to the $\varphi$-operator on
 $D^{\prime}$. Similarly, let $\Fil^{\prime}$ denote the filtration on
 $D\otimes_{M_0}M$ corresponding under $j$ to the filtration on
 $D^{\prime}\otimes_{M_0}M$. Choose an isomorphism of
 $E[\varepsilon]\otimes_{\bb{Q}_l}M_0$ modules between $D$ and
 $D_{E}\otimes_{E}E[\varepsilon]$ which identifies $D(\alpha)$ with
 $D(\alpha_E)\otimes 1$.
Let $\mf{g}_{D_E}$ and $\mf{g}_{D_E}^{\circ}$
 denote the Lie algebras of $\GSp(D_{E},D(\alpha_E))$ and
 $\Sp(D_E,D(\alpha_E))$ respectively. Similarly, let $\mf{g}_{D_{E,M}}$
 denote the Lie algebra of $\GSp(D_{E}\otimes_{M_0}M,
 \alpha_{E}\otimes 1)$. Let $\mf{b}_{D_{E,M}}$ denote the Lie algebra of
 the Borel subgroup of $\GSp(D_{E}\otimes_{M_0}M,\alpha_E \otimes 1)$
 which stablises the filtration on $D_{E}\otimes_{M_0}M$.
 Then there exists $X \in \mf{g}_{D_E}^{\circ}$ and
 $Y \in \mf{g}_{D_{E,M}}$ such that $\varphi^{\prime}=(1+\varepsilon
 X)\varphi$ and $\Fil^{\prime} = (1+\varepsilon Y)\Fil$. Moreover, any
 such pair $X,Y$ gives rise to a deformation of $(D_{E},D(\alpha_E))$
 and we get a surjective linear map
\[ \mf{g}_{D_E}^{\circ} \oplus \mf{g}_{D_{E,M}}/\mf{b}_{D_{E,M}} \onto D_{D_E}(E[\varepsilon]).\]
The kernel of this map is the image of the map
\[ \mf{g}_{D_E}^{\circ} \rightarrow  \mf{g}_{D_E}^{\circ} \oplus
\mf{g}_{D_{E,M}}/\mf{b}_{D_{E,M}} \]
sending $Z$ to the pair $(Z - \varphi \circ Z \circ \varphi^{-1},
Z)$.
Denote the
kernel of this last map by
$(\mf{g}_{D_E}^{\circ})^{\varphi=1,\Fil}$. We have shown that
\[ d = \dim_E \left( \mf{g}^{\circ}/(\mf{g}^{\circ})^{G_M}\right) +
\dim_{E} \mf{g}_{D_{E,M}}/\mf{b}_{D_{E,M}} + \dim_E (\mf{g}_{D_E}^{\circ})^{\varphi=1,\Fil}.\]
The result now follows from the fact that $\dim_E \mf{g}^{\circ}=10$,
$\dim_{E} \mf{g}_{D_{E,M}}/\mf{b}_{D_{E,M}}=4[M:\bb{Q}_l]$ and
  $
(\mf{g}^{\circ})^{G_M} \cong  
(\mf{g}_{D_E}^{\circ})^{\varphi=1,\Fil}$ via $D_{\cris}$.

Now suppose that $p\neq l$. In this case we only need to establish a
lower bound on the dimension, and we do this by means of a slight
variant of Mazur's lower bound for the dimension of an unrestricted
deformation ring (see Proposition 2 of \cite{MR1012172}). Note that by
the construction of the ring $R^{sympl,\tau,\psi}_{\rhobar}$, we need only show
that each irreducible component of $R^{\sq,sympl,\psi}_{\rhobar}$ has dimension at
least 11.

Let $\m^{sympl}$ denote the maximal ideal of $R^{\sq,sympl,\psi}_{\rhobar}$. Then
$R^{\sq,sympl,\psi}_{\rhobar}$ is the quotient of a power series ring over $\bigO$ in
$\dim_k \m^{sympl}/((\m^{sympl})^2,\pi_K)$ variables. The argument of
the proof of Lemma 4.1.1 of \cite{MR2459302} shows that it is
necessary to quotient out by at most $\dim_k H^2(G_M,\ad^0\rhobar)$
relations. Thus every component of $R^{\sq,sympl,\psi}_{\rhobar}$ has
dimension at least \[1+\dim_k \m^{sympl}/((\m^{sympl})^2,\pi_K) -
\dim_k H^2(G_M,\ad^0\rhobar).\] Now, $
\m^{sympl}/((\m^{sympl})^2,\pi_K)$ is dual to the tangent
space \[D^{\sq,sympl}(k[\epsilon]/(\epsilon^2)),\]where $D^{\sq,sympl}$
is the functor represented by $R^{\sq,sympl,\psi}_{\rhobar}$. The elements of this
space are $1$-cocycles in $Z^1(G_M,\ad^0\rhobar)$, so we see that \begin{align*}\dim_k
\m^{sympl}/((\m^{sympl})^2,\pi_K)&=\dim_k Z^1(G_M,\ad^0\rhobar)\\ & =
\dim_k H^1(G_M,\ad^0\rhobar)+\dim_k\ad^0\rhobar-\dim_k
H^0(G_M,\ad^0\rhobar).\end{align*} Thus every component of $R^{\sq,sympl,\psi}_{\rhobar}$ has
dimension at least \begin{align*}1+\dim_k H^1(G_M,\ad^0\rhobar)+\dim_k\ad^0\rhobar-\dim_k
H^0(G_M,\ad^0\rhobar) -\dim_k H^2(G_M,\ad^0\rhobar)\\ = 1+\dim_k
\ad^0\rhobar\\=11\end{align*} by the local Euler
characteristic formula, as required.
\end{proof}

\begin{remark}
  It is presumably possible to use the techniques of \cite{kisinpst} to prove that if $p\neq l$, and $\tau$ is
  such that the ring $R^{sympl,\tau,\psi}$ is non-zero, then it is
  equidimensional of dimension $11$. As we do not need this result we
  have not attempted to verify this.
\end{remark}

The following lemma can be proved in exactly the same way as Lemma 3.3.3
of \cite{ger}.

\begin{lem}
\label{lem: dimension of symplectic ordinary crystalline rings}
  Let $M$ be a finite extension of $\bb{Q}_l$. There is a quotient $R^{sympl,\triangle_{\lambda},cr,\psi}_{\rhobar}$ of
  $R^{sympl,\mathbf{v}_{\lambda},cr,\psi}_{\rhobar}$ corresponding to a union of
  irreducible components such that for any finite local $K$-algebra $B$, a homomorphism of $\mc{O}$-algebras $\zeta :
  R^{sympl,\mathbf{v}_{\lambda},cr,\psi}_{\rhobar} \rightarrow B$ factors through
  $R^{sympl,\triangle_{\lambda},cr,\psi}_{\rhobar}$ if and only if $\zeta \circ
  \rho^{\square}$ is ordinary of weight $\lambda$.
\end{lem}

\subsection{A lower bound on the dimension of a symplectic deformation
  ring}In this section we outline a proof of a lower bound on the
dimension of a global deformation ring. This material is by now rather
standard (see for example section 4 of \cite{MR2459302} or Corollary
2.3.5 of \cite{cht}), and we
content ourselves with a sketch of the proofs.

Suppose that $F^+$ is a totally real field. Let
$\rhobar:G_{F^+}\to\GSp_4(k)$ be absolutely irreducible. Let $S$ be a
finite set of finite places of $F^+$, including all places at which
$\rhobar$ is ramified, and all places dividing $l$. Let $F^+(S)$
be the maximal extension of $F^+$ unramified outside of $S$, and write
$G_{F^+(S)/F^+}$ for the Galois group $\Gal(F^+(S)/F^+)$ (note that we
do not use the more usual notation $G_{F^+,S}$ for this group, as to
do so would be inconsistent with \cite{cht} and the earlier sections
of this paper). Fix a character
$\psi: G_{F^+(S)/F^+}\to\bigO^\times$ lifting the character
$\mu\circ\rhobar$. If $R$ is a complete local Noetherian
$\bigO$-algebra with residue field $k$, then an $R$-valued deformation
of $\rhobar$ is a $\ker(\GSp_4(R)\to\GSp_4(k))$-conjugacy class of
liftings of $\rhobar$ to $\GSp_4(R)$. Since $\rhobar$ is absolutely
irreducible, it is an easy consequence of Schur's lemma that $\rhobar$
has a universal symplectic deformation with fixed similitude factor
$\psi$ to a complete local Noetherian $\mc{O}$-algebra
$R_{F^+,S}^{sympl,\psi}$ (see for example Theorem 3.3 of
\cite{MR1643682}).

Let $R_{F^+,S}^{\sq,sympl,\psi}$ denote the
complete local Noetherian $\bigO$-algebra representing the functor
$\mathcal{D}_{F^+,S}^{\sq,sympl,\psi}$ which assigns to a complete
local Noetherian $\bigO$-algebra $R$ with residue field $k$ the set of
equivalence classes of tuples $(\rho,\{\alpha_v\}_{\{v\in S\}})$ where
$\rho$ is a lifting of $\rhobar$ to $R$ with similitude character
$\psi$ and for each $v\in S$,
$\alpha_v\in\ker(\GSp_4(R)\to\GSp_4(k))$. Two such tuples $(\rho,\{
\alpha_v\}_{ \{ v \in S\}})$ and $(\rho^{\prime},\{
\alpha^{\prime}_v\}_{ \{ v \in S\}})$ are said to be equivalent if
there exists an element $\beta \in \ker(\GSp_4(R)\rightarrow
\GSp_4(k))$ with $\rho^{\prime}=\beta \rho \beta^{-1}$ and
$\alpha^{\prime}_v=\beta \alpha_v$ for all $v \in S$.
Note that $R_{F^+,S}^{\sq,sympl,\psi}$ is formally smooth
over $R_{F^+,S}^{sympl,\psi}$ of relative dimension $11|S|-1$. 
For each $v\in S$ let $R_v^{\sq,sympl,\psi}$ denote the universal
$\mc{O}$-lifting ring for symplectic liftings of $\rhobar|_{G_{F^+_v}}$ with
similitude character $\psi$. Let $R_S^\psi=\widehat{\otimes}_{v\in
  S}R_v^{\sq,sympl,\psi}$. There is a natural map $R^\psi_S\to
R_{F^+,S}^{\sq,sympl,\psi}$ given on $R$-points by sending a tuple
$(\rho,\{\alpha_v\}_{\{v\in S\}})$ to the tuple
$(\alpha_v^{-1}\rho|_{G_{F^+_v}}\alpha_v)_{\{v\in S\}}$ (note that this map
is well-defined by the definition of equivalence for these tuples).

For $i=1$, $2$ we let $h^i_S$ denote the
$k$-dimension of the kernel of the natural
map \[H^i(G_{F^+(S)/F^+},\ad^0\rhobar)\to\prod_{v\in
  S}H^i(G_{F^+,v},\ad^0\rhobar).\] Let $\mf{m}_{F^+,S}$ denote the maximal ideal of $R^{\sq,sympl,\psi}_{F^+,S}$, and
$\mf{m}_{S}$ the maximal ideal of $R^\psi_S$.

\begin{prop}\label{prop: abstract dimension bound for global rings}Let \[\eta:\mf{m}_S/(\mf{m}_S^2,\pi_K)\to\mf{m}_{F^+,S}/(\mf{m}_{F^+,S}^2,\pi_K)\]be
  the natural map. Then $R^{\sq,sympl}_{F^+,S}$ is  a quotient of a
  power series ring over $R_S^\psi$ in $\dim_k\coker\eta$ variables by at
  most $\dim_k\ker\eta+h^2_S$ relations.
  
\end{prop}
\begin{proof}
  This may be proved in exactly the same fashion as Proposition 4.1.4
  of \cite{MR2459302}.
\end{proof}

\begin{cor}\label{cor: global symplectic estimate for totally real}
 Suppose that $H^0(G_{F^+(S)/F^+},(\ad^0\rhobar)^*(1))=0$. Let $s=\sum_{v|\infty}\dim_k H^0(G_{F^+,v},\ad^0\rhobar)$. Then for some
 non-negative integer $r$ and some $f_1,\dots,f_{r+s}$, there is an
 isomorphism \[R_{F^+,S}^{\sq,sympl,\psi}\isoto R_S^\psi[[x_1,\dots,x_{r+|S|-1}]]/(f_1,\dots,f_{r+s}).\]
\end{cor}
\begin{proof} This is very similar to the proof of Proposition 4.1.5
  of \cite{MR2459302}. By Proposition \ref{prop: abstract dimension
    bound for global rings} we see that the result will hold with $s$
  chosen such
  that \[|S|-s-1=\dim_k\mf{m}_{F^+,S}/(\mf{m}_{F^+,S}^2,\pi_K)-\dim_k\mf{m}_S/(\mf{m}_S^2,\pi_K)-h^2_S,\]so
  it suffices to show that this agrees with the value of $s$ in the statement
  of the corollary. Note firstly that
  $\Hom_k(\mf{m}_{F^+,S}/(\mf{m}_{F^+,S}^2,\pi_K),k)$ is naturally
  isomorphic to
  $\mathcal{D}_{F^+,S}^{\sq,sympl,\psi}(k[\epsilon]/(\epsilon^2))$. Consideration
  of the equivalence relation defining
  $\mathcal{D}_{F^+,S}^{\sq,sympl,\psi}$ shows that this space has
  $k$-dimension \[11|S|+\dim_kH^1(G_{F^+(S)/F^+},\ad^0\rhobar)-\dim_kH^0(G_{F^+(S)/F^+},\ad^0\rhobar)-1.\]
  Similarly, \begin{align*}\dim_k\mf{m}_S/(\mf{m}_S^2,\pi_K)&=\sum_{v\in
    S}(\dim\ad^0\rhobar+\dim_kH^1(G_{F^+_v},\ad^0\rhobar)-\dim_kH^0(G_{F^+_v},\ad\rhobar))\\
  &=\sum_{v\in
    S}(10+\dim_kH^1(G_{F^+_v},\ad^0\rhobar)-\dim_kH^0(G_{F^+_v},\ad^0\rhobar)).\end{align*}

The condition that $H^0(G_{F^+(S)/F^+},(\ad^0\rhobar)^*(1))=0$, together
with the last 3 terms of the Poitou-Tate sequence, shows that the map
$\theta^2$ is surjective, so
that \[h^2_S=\dim_kH^2(G_{F^+(S)/F^+},\ad^0\rhobar)-\sum_{v\in
  S}\dim_kH^2(G_{F^+_v},\ad^0\rhobar).\]Thus \[\dim_k\mf{m}_{F^+,S}/(\mf{m}_{F^+,S}^2,\pi_K)-\dim_k\mf{m}_S/(\mf{m}_S^2,\pi_K)-h^2_S=|S|+\sum_{v\in
  S}\chi(G_{F^+_v},\ad^0\rhobar)-\chi(G_{F^+(S)/F^+},\ad^0\rhobar)-1,\]where $\chi$
denotes the Euler characteristic as a $k$-vector space, and it
suffices to show that \[\sum_{v\in
  S}\chi(G_{F^+_v},\ad^0\rhobar)-\chi(G_{F^+(S)/F^+},\ad^0\rhobar)=\sum_{v|\infty}\dim_k
H^0(G_{F^+_v},\ad^0\rhobar).\]This follows at once from the local and global
Euler characteristic formulae.\end{proof}

 For each place $v\in S$ not dividing $l$ we fix a type $\tau_v$ such that
$\rhobar|_{G_{F^+_v}}$ has a symplectic lifting of type $\tau_v$ and similitude
character $\psi|_{G_{F^+_v}}$, and we fix a quotient $R_v$  of
$R^{sympl,\tau_v,\psi}_{v}$ corresponding to a union of irreducible components. For each
$v|l$ we fix a weight $\lambda_v$ such that $\rhobar|_{G_{F^+_v}}$
has a crystalline symplectic lift of weight $\lambda_v$ and similitude
character $\psi|_{G_{F^+_v}}$, and we fix a quotient $R_v$ of
$R_{v}^{sympl,\Delta_{\lambda_v},cr,\psi}$
corresponding to a union of irreducible components. Let
$R_S^{\psi,\tau}:=\widehat{\otimes}_{v\in S}R_v$, and let
$R_{F^+,S}^{\sq,\tau,\psi}=R_{F^+,S}^{\sq,sympl,\psi}\widehat{\otimes}_{R_S^\psi}R_S^{\psi,\tau}$. Let
$R_{F^+,S}^{sympl,\tau,\psi}$ be the universal deformation $\bigO$-algebra
representing the functor which assigns to $R$ the
$\ker(\GSp_4(R)\to\GSp_4(k))$-conjugacy classes of liftings of
$\rhobar$ with the property that for each $v\in S$ the corresponding
lifting of $\rhobar|_{G_{F^+_v}}$ gives an $R$-point of $R_v$ (that this functor is well defined follows from the symplectic analogue of Lemma \ref{lem: components are conjugation invariant} which can be proved in the same way). Thus
$R^{\sq,\psi,\tau}_{F^+,S}$ is formally smooth over $R^{sympl,\psi,\tau}_{F^+,S}$ of
relative dimension $11|S|-1$.

\begin{defn}
  We say that $\rhobar$ is \emph{odd} if for all complex conjugations
  $c\in G_{F^+}$, $(\mu\circ\rhobar)(c)=-1$.
\end{defn}

\begin{prop}\label{symplectic def ring has dimension at least one.}Assume that $\rhobar$ is odd and that
  $H^0(G_{F^+(S)/F^+},(\ad^0\rhobar)^*(1))=0$. Then the Krull dimension of
  $R^{sympl,\psi,\tau}_{F^+,S}$ is at least one.
  
\end{prop}
\begin{proof}
  It suffices to check that the dimension of
  $R^{\sq,\psi,\tau}_{F^+,S}$ is at least $11|S|$. By Corollary
  \ref{cor: global symplectic estimate for totally real}, it would be
  enough to check that
  \[\dim R_S^{\psi,\tau}+|S|-1-\sum_{v|\infty}\dim_k
  H^0(G_{F^+_v},\ad^0\rhobar)\ge 11|S|.\] By Propositions \ref{prop:
    dimension of local symplectic deformation rings} and \ref{lem:
    dimension of symplectic ordinary crystalline rings}, \[\dim R_S^{\psi,\tau}\geq
  1+10|S|+4[F^+:\Q].\]An easy calculation using the fact that
  $\rhobar$ is odd shows that for each $v|\infty$,
  $\dim_kH^0(G_{F^+_v},\ad^0\rhobar)=4$ (for example, one easily checks that
  if $c_v$ is a corresponding complex conjugation then
  $\rhobar(c_v)$ is conjugate to the diagonal matrix
  $diag(1,1,-1,-1)$, and one may then compute explicitly). Thus   \[\dim R_S^{\psi,\tau}+|S|-1-\sum_{v|\infty}\dim_k
  H^0(G_{F^+_v},\ad^0\rhobar)\ge 10|S|+4[F^+:\Q]+|S|-4[F^+:\Q]= 11|S|,\]as required.
\end{proof}

\subsection{Relationship to unitary representations}Let $F$ be a
totally imaginary CM field with maximal totally real field $F^+$, with
the property that all primes in $S$ split in $F$. Let $\tilde{S}$
denote a set of places of $F$ consisting of one place dividing each
place in $S$. Recall that we let $G_{F^+,S}=\Gal(F(S)/F^+)$. Let
$\rho:G_{F^+}\to\GSp_4(R)$ be a continuous representation, with $R$ a
complete local Noetherian ring. Then, as in Lemma 2.1.2 of \cite{cht},
there is a continuous homomorphism $r: G_{F^+}\to\G_4(R)$ determined
by \[r(g)=(\rho(g),(\mu\circ\rho)(g))\]if $g\in G_{F}$,
and \[r(g)=(\rho(g)J^{-1},-(\mu\circ\rho)(g))j\]if $g\notin G_{F}$. We
have \[\nu\circ r=\mu\circ\rho.\] Furthermore, this construction is
obviously compatible with deformations, in the sense that if
$B\in\ker(\GSp_4(R)\to\GSp_4(k))$ and $\rho$ is replaced by $\rho_B$
with \[\rho_B(g):=B\rho(g)B^{-1},\]then $r$ is replaced by $r_B$
with \[r_B(g):=(aB,1)r(g)(aB,1)^{-1},\]where $a^2=\mu(B)^{-1}$ (such
an $a$ exists because $\mu(B)\in 1+\mf{m}_B$ and $l>2$). Applying this
construction to the universal symplectic
deformation of the previous section \[\rho^{univ}:G_{F^+(S)/F^+}\to\GSp_4(R^{sympl,\psi,\tau}_{F^+,S}),\]
we obtain a
deformation \[r^{sympl}:G_{F^+(S)/F^+}\to\G_4(R^{sympl,\psi,\tau}_{F^+,S}).\]
We may also consider the corresponding residual
representation \[\rbar:G_{F^+,S}\to\G_4(k),\] and (in the notation of
sections 2.2 and 2.3 of \cite{cht}) the deformation
problem \[\mathcal{S}=(F/F^+,S,\tilde{S},\bigO,\rbar,\psi,R_v)\]with
corresponding universal deformation \[r_{\mc{S}}:G_{F^+,S}\to\G_n(R_{\mc{S}}^{\univ}).\]  Since
$G_{F^+(S)/F^+}$ is a quotient of $G_{F^+,S}$, there is a homomorphism
$\theta:R_{\mc{S}}^{\univ}\to
R^{sympl,\psi,\tau}_{F^+,S}$ such that there is an equality of
deformations \[r^{sympl}=\theta\circ r_{\mc{S}}.\]
\begin{lem}\label{lem:symplectic finite over unitary}
$R^{sympl,\psi,\tau}_{F^+,S}$ is finite over $R_{\mc{S}}^{\univ}$.
\end{lem}
\begin{proof}
 Let $\rho_{F,F^+}$ denote the
  $\GSp_4(R^{sympl,\psi,\tau}_{F^+,S}/\theta(\m_{R_{\mc{S}}^{\univ}}))$-valued representation obtained
  from $\rho^{univ}$, and let $r_{F,F^+}$ denote the corresponding
  representation to $\G_4(R^{sympl,\psi,\tau}_{F^+,S}/\theta(\m_{R_{\mc{S}}^{\univ}}))$. Then
  $r_{F,F^+}$ is equivalent to $\rbar$, so it has finite image, and
  thus the image of $\rho_{F,F^+}$ is also finite. An
  argument exactly as in the proof of Lemma \ref{lem: deformation ring is finite over another for the
  unitary group} (using Lemma \ref{lem: symplectic representation
  valued over trace} to see that the universal deformation ring is
generated by traces) shows that  $R^{sympl,\psi,\tau}_{F^+,S}/\theta(\m_{R_{\mc{S}}^{\univ}})$ is
finite, as required.
\end{proof}

\subsection{Companion forms for symplectic Galois representations and
  automorphic representations for  $\GL_4$}We now prove our first
companion forms theorem for symplectic representations. This theorem
applies to automorphic representations of $\GL_4$; in the next section
we will use functoriality to deduce a result for automorphic
representations of $\GSp_4$.

Suppose that $\pi$ is a RAESDC representation of $\GL_4(\A_{F^+})$,
with $\pi^\vee\cong\chi\pi$. Let $\iota:\Qlbar\isoto\C$.  Then there
is a continuous semisimple
representation $$\rho_{l,\iota}(\pi):G_{F^+}\to\GL_4(\Qlbar)$$
associated to $\pi$ (see theorem 1.1 of \cite{BLGHT}). We say that a representation
$\rho:G_{F^+}\to\GL_4(\Qlbar)$ is automorphic if
$\rho\cong\rho_{l,\iota,}(\pi)$ for some $\iota$, $\pi$.

The representation $\rhobar_{l,\iota}(\pi)$ may be conjugated to be valued in the
ring of integers of a finite extension of $\Ql$, and we may reduce it
modulo the maximal ideal of this ring of integers and semisimplify to
obtain a well-defined continuous
representation $$\rhobar_{l,\iota}(\pi):G_{F^+}\to\GL_4(\Flbar).$$We
say that a representation $\rhobar:G_{F^+}\to\GL_4(\Flbar)$ is automorphic if
$\rhobar\cong\rhobar_{l,\iota}(\pi)$ for some $\iota$, $\pi$. We say that
$\rhobar$ is symplectic ordinarily automorphic if
$\rhobar\cong\rhobar_{l,\iota}(\pi)$, where $\pi$ is $\iota$-ordinary and $\rho_{l,\iota}(\pi)$ is symplectic. We say that
$\rhobar$ is symplectic ordinarily automorphic of level prime to $l$
if furthermore $\pi$ may be taken to be unramified at all places
dividing $l$.

\begin{cor}\label{cor: o rank one for symplectic deformation ring}
 Assume that $\rhobar$ is symplectic ordinarily automorphic, that $\rhobar(G_{F^+(\zeta_l)})$ is
 big, and that $(\overline{F^+})^{\ker \ad\rhobar}$ does not contain $F^+(\zeta_l)$. Then $R^{sympl,\psi,\tau}_{F^+,S}$ is a finite $\bigO$-module of
  rank at least one.
\end{cor}
\begin{proof}This follows from Lemma \ref{lem:symplectic finite over
    unitary}, Corollary \ref{cor: universal unitary ring has o-rank at
    least one} and Proposition \ref{symplectic def ring has dimension
    at least one.} (note we are free to choose $F$ linearly disjoint
  from $(\overline{F^+})^{\ker \ad \rhobar}(\zeta_l)$ over $F^+$).
\end{proof}

Suppose that $\rho:G_{F^+}\to\GSp_4(\mc{O}_{\Qlbar})$ is crystalline. Then the
similitude factor $\psi$ of $\rho$ is a crystalline character of
$G_{F^+}$, so there is an integer $n$ such that for all places $v|l$,
$\psi|_{I_{F^+_v}}=\epsilon^n$. Suppose now that
$\rho':G_{F^+}\to\GSp_4(\mc{O}_{\Qlbar})$ is another crystalline representation
with similitude factor $\psi'$, and that $\rhobar=\rhobar'$. Then
$\psibar=\psibar'$, and there is an integer $n'$ such that for all places $v|l$,
$\psi|_{I_{F^+_v}}=\epsilon^{n'}$. Thus $\epsilon^{n'-n}$ is a
crystalline character of $G_{F^+}$ whose reduction mod $l$ is
everywhere unramified. This motivates the choice of similitude factor
in the following theorem (in particular, it shows that our choice of
similitude factor does not exclude any possibilities for the
Hodge-Tate weights of the Galois representations we construct).

\begin{thm}\label{companion forms symplectic and gl4} Let $F^+$ be a totally real field. Let $l\ge 5$ be a prime
  number such that $[F^+(\zeta_l):F^+]>2$. Suppose that \[\rhobar:G_{F^+}\to\GSp_4(\Flbar)\] is an
  irreducible representation, and let $n$ be an integer such that
  $\overline{\epsilon}^n$ is an unramified character of $G_{F^+}$. Suppose that $\rhobar$ satisfies the
following assumptions.
  \begin{enumerate}
   \item There are finite fields $\F_l \subset k\subset k'$ such that
     $\Sp_4(k)\subset\rhobar(G_{F^+})\subset (k')^\times\GSp_4(k)$.
  \item The representation $\rhobar$ is symplectic ordinarily
    automorphic of level prime to $l$; say $\rhobar\cong\rhobar_{l,\iota}(\pi)$, and write
    $\psi$ for the similitude factor of
    $\rho_{l,\iota}(\pi)$.
  \item Define $\psi_n:=\psi\epsilon^n\tilde{\omega}^{-n}$, where
    $\tilde{\omega}$ is the Teichm\"{u}ller lift of the mod $l$
    cyclotomic character (so $\overline{\psi}_n=\overline{\psi}$, and
    $\psi_n$ is crystalline). There is an element $\lambda\in(\Z^4)^{\Hom(F^+,\Qlbar)}$ such that
    \begin{itemize}
   \item for all $\tau\in\Hom(F^+,\Qlbar)$, we have
      $\lambda_{\tau,1}\geq\dots\geq\lambda_{\tau,4}$,
    \item for every place $v|l$ of $F^+$, $\rhobar|_{G_{F^+_v}}$ has an 
      ordinary crystalline symplectic lift of weight
      $(\lambda_{\tau})_{\tau}$ (where the
      indexing set runs over the embeddings $\tau\in\Hom(F^+,\Qlbar)$
      inducing $v$) and similitude factor $\psi_n$.
    \end{itemize}
  \end{enumerate}
  Then $\rhobar$ has an ordinary crystalline symplectic lift $\rho$ of weight
  $\lambda$ which is ordinarily automorphic of level prime to $l$.

 Given any finite set of places $S$ of $F^+$, and an inertial type $\tau_v$ for each $v\in S$ not dividing $l$ such that
$\rhobar|_{G_{F^+_v}}$ has a symplectic lift of type $\tau_v$ and
similitude factor $\psi_n$, $\rho$ can be
chosen to be of similitude factor $\psi_n$ and of type $\tau_v$ at $v$ for all places $v\in S$,
$v\nmid l$. More precisely, given a choice of a component of each ring
$R^{sympl,\tau_v,\psi_n}$ ($v\in S$, $v\nmid
l$) and  $R_{\rhobar|_{G_{F^+_v}}}^{sympl,\triangle_\lambda,cr,\psi_n}$
($v|l$), $\rho$ may be chosen so as to
give a point on each of these components.
\end{thm}

\begin{proof}It suffices to prove the last statement. Enlarge $S$ if
  necessary so that $S$ contains all places of $F^+$ dividing $l$ and
  all places at which $\rhobar$ is ramified. Choose a totally
  imaginary quadratic extension $F/F^+$ such that all places in $S$ split in
  $F$, and such that $F$ is linearly disjoint from
  $(\overline{F^+})^{\ker\rhobar}$. Note that we can choose a type
  $\tau_v$ at any place not dividing $l$ such that
  $\rhobar|_{G_{F^+_v}}$ has a symplectic lift of type $\tau_v$ and
  similitude factor $\psi_n$; $\rho_{l,\iota}(\pi)|_{G_{F^+_v}}$
  provides such a lift if $n=0$, and we may twist this lift in the
  general case (note that $\psi_n\psi^{-1}$ is unramified at $v$, and
  there is no obstruction to taking a square root of an unramified
  character). We then consider deformation problems as in the previous
  section. By Lemma 2.5.5 of \cite{cht}, and the fact that
  $\operatorname{PSp}_4(k)$ is simple, we see that
  $\rhobar(G_{F^+(\zeta_l)})$ is big. Again, because
  $\operatorname{PSp}_4(k)$ is simple, the abelianisation of
  $\ad\rhobar(G_{F^+})$ is a subgroup
  of \[\operatorname{PGSp}_4(k)/\operatorname{PSp}_4(k)\isoto
  k^\times/(k^\times)^2.\]As this latter group has cardinality $2$ and
  $[F^+(\zeta_l):F^+]>2$, we see that $\overline{F^+}^{\ker\ad\rhobar}$ does not
  contain $F^+(\zeta_l)$. Then Corollary \ref{cor: o rank one for
    symplectic deformation ring} gives the existence of a Galois
  representation satisfying every property except possibly
  automorphicity, which follows from Theorem \ref{thm: R=T} and Lemmas
  1.4 and 1.5 of \cite{BLGHT}.
\end{proof}

\subsection{Companion forms for $\GSp_4$}We now prove results for
automorphic representations for $\GSp_4$ over totally real fields by
making use of known cases of functoriality between $\GSp_4$ and
$\GL_4$. The main result we need is the following.

\begin{thm}\label{transfer from gsp4 to gl4}
  Let $M$ be a number field. There is an injective map $\pi\mapsto\Pi\boxtimes\theta$ from the set
  of globally generic cuspidal representations $\pi$ of $\GSp_4$ over $M$ to the set
  of globally generic representations $\Pi\boxtimes\theta$ of
  $\GL_4\times\GL_1$ over $M$. This map has the following properties:
  \begin{enumerate}
  \item $\theta=\omega_\pi$ (the central character of $\pi$), and the central character of $\Pi$ is $\omega_\pi^2$.
  \item $\Pi\cong\Pi^\vee\otimes\omega_\pi$.
  \item For each place $v$ of $M$ there is an equality of
   Weil-Deligne representations $\rec(\pi_v)=\rec(\Pi_v)$,
   where we also denote the local Langlands correspondence of
   \cite{gantakeda} by $\rec$, and consider $\GSp_4$ as a subgroup of
   $\GL_4$.
  \item If $\Pi\boxtimes\theta$ is such that $\Pi$ is cuspidal, then
    $\Pi\boxtimes\theta$ is in the image of the map if and only if the
    partial $L$-function $L^S(s,\Pi,\bigwedge^2\otimes\theta^{-1})$ has a pole at
    $s=1$ (where $S$ is any finite set of places of $M$).
  \item If $\Pi\boxtimes\theta$ is in the image of the map and $\Pi$
    is not cuspidal, then $\Pi$ is an isobaric direct sum of two
    cuspidal representations of $\GL_2$.
  \end{enumerate}
\end{thm}
\begin{proof}
  This is a special case of Theorem 13.1 of \cite{gantakeda}.
\end{proof}

\begin{defn}Let $F^+$ be a totally real field, and let $\pi$ be a
  cuspidal automorphic representation of $\GSp_4$ over
  $F^+$. Assume further that $\pi$ is automorphic of weight
  $\mu=(\mu_{v,1},\mu_{v,2};\alpha_v)_{v|\infty} \in (\bb{Z}^3)^{\Hom(F^+,\bb{R})}$, in the sense that for each $v|\infty$, $\pi_v$ is a
  discrete series representation with the same central and
  infinitesimal characters as the finite-dimensional irreducible
  algebraic representation of highest weight given
  by \[t=\diag(t_1,t_2,t_3,t_4)\mapsto
  t_1^{\mu_{v,1}}t_2^{\mu_{v,2}}\mu(t)^{-(\mu_{v,1}+\mu_{v,2}+\alpha_v)/2}.\]Here
  $\mu_{v,1}\ge\mu_{v,2}\ge 0$ and $\mu_{v,1}+\mu_{v,2}$ has the same parity
  as $\alpha_v$. Fix an isomorphism $\iota:\Qlbar\isoto\C$. Then we say that
  there is a Galois representation associated to $\pi$ if there is a continuous semisimple
  representation \[\rho_{\pi,\iota}:G_{F^+}\to \GSp_4(\Qlbar)\]such that:
  \begin{itemize}
  \item  for each finite place $v\nmid l$, \[\iota
  WD(\rho_{\pi,\iota}|_{W_{F^+_v}})^{ss}\cong\rec(\pi_v\otimes|\cdot|^{-3/2})^{ss},\]where
  $\rec$ is the local Langlands correspondence of \cite{gantakeda} and $|\cdot|$ is the composition of the similitude character and the norm character.
  \item If $\pi_v$ is unramified at a place $v|l$ then
    $\rho_{\pi,\iota}$ is crystalline at $v$, and in any case it is de Rham.
  \item Define $\lambda_{\iota,\mu}\in (\Z^4_+)^{\Hom(F^+,\Qlbar)}$ by
    letting \[\lambda_{\iota,\mu,\tau}=(\delta_{\iota\circ\tau}+\mu_{\iota\circ\tau,1}+\mu_{\iota\circ\tau,2},\delta_{\iota\circ\tau}+\mu_{\iota\circ\tau,1},\delta_{\iota\circ\tau}+\mu_{\iota\circ\tau,2},\delta_{\iota\circ\tau})\]for
    each embedding $\tau:F^+\into\Qlbar$,
    where \[\delta_v:=-\frac{1}{2}(\mu_{v,1}+\mu_{v,2}+\alpha_v)\]
    for each $v|\infty$.
    Then for each  $\tau:F^+\into\Qlbar$ lying over a place $v$ of
    $F^+$, the Hodge-Tate weights of $\rho_{\pi,\iota}|_{G_{F^+_v}}$ with respect
    to $\tau$ are the $\lambda_{\iota,\mu,\tau,j}+4-j$.
  \end{itemize}
\end{defn}
We now define what it means for a cuspidal automorphic representation
of $\GSp_4$ to be ordinary. We could do this directly in terms of
Hecke operators on $\GSp_4$, but for the sake of brevity we use the
local Langlands correspondences for $\GL_4$ and $\GSp_4$ and the
definition of ordinarity for $\GL_4$.
\begin{defn}
  Let $\iota:\Qlbar\isoto\C$ be an isomorphism, and let $\pi$ be a
  cuspidal automorphic representation of $\GSp_4(\A_{F^+})$ which is
  of weight $\mu=(\mu_{v,1},\mu_{v,2};\alpha_v)_{v|\infty}$ in the
  above sense. Let $\lambda \in (\bb{Z}^4_+)^{\Hom(F^+,\bb{R})}$ be
  defined by $\lambda_v=(\delta_v +
  \mu_{v,1}+\mu_{v,2},\delta_v+\mu_{v,1},\delta_v+\mu_{v,2},\delta_v)$. We
  say that $\pi$ is $\iota$-ordinary if for each $v|l$, the
  irreducible admissible representation $\Pi_v$ of $\GL_4(F^+_v)$
  with \[\rec(\pi_v)=\rec(\Pi_v)\]satisfies
  $(\iota^{-1}\Pi_v)^{\ord}\neq 0$, where the space $(\iota^{-1}\Pi_v)^{\ord}$ is
  defined as in section \ref{subsec: ordinary
    automorphic representations}.
\end{defn}

\begin{defn}
  We say that a continuous irreducible
  representation \[\rho:G_{F^+}\to \GSp_4(\Qlbar)\] is
  $\GSp_4$-automorphic (of weight $\lambda\in (\Z^4_+)^{\Hom(F^+,\Qlbar)}$) if there is
  a $\pi$ and an
  $\iota:\Qlbar\isoto\C$ with $\rho\cong\rho_{\pi,\iota}$, and for
  each $\tau:F^+\into\Qlbar$ lying over a place $v$ of $F^+$, the
  Hodge-Tate weights of $\rho_{\pi,\iota}$ with respect to $\tau$ are
  the $\lambda_{\tau,j}+4-j$. By the above definitions, we see that
  this is equivalent to $\pi$ being automorphic of weight $\mu$ with
  $\lambda_{\iota,\mu}=\lambda$. We say that $\rho$ is $\GSp_4$-automorphic
  and holomorphic if $\pi$ can be chosen to be a holomorphic discrete
  series at all infinite places, and that $\rho$ is
  $\GSp_4$-automorphic and generic if $\pi$ can be chosen to be
  globally generic (note that it is possible for $\rho$ to be both
  holomorphic and generic, corresponding to different choices of $\pi$
  in the same global $L$-packet). We say that $\rho$ is
  $\GSp_4$-ordinarily automorphic if $\pi$ can be chosen to be
  $\iota$-ordinary. We say that $\rho$ is
  $\GSp_4$-ordinarily automorphic and holomorphic (respectively generic)
  if $\pi$ may be chosen to be simultaneously $\iota$-ordinary and holomorphic discrete series at all infinite places
  (respectively globally generic). Finally, we say in addition that
  $\rho$ is automorphic of level prime to $l$ if $\pi_l$ is
  unramified.
\end{defn}

In recent work (\cite{sorensen}) Sorensen has used Theorem
\ref{transfer from gsp4 to gl4} and the constructions of
\cite{MR1876802} to associate Galois representations to certain
globally generic cuspidal representations of $\GSp_4$ over totally
real fields. In particular, he obtains the following theorem, which
gives a ready supply of Galois representations associated to
automorphic representations of $\GSp_4$.

\begin{thm}\label{sorensen main theorem on associating galois representations}Let $F^+$ be a totally real field, and let $\pi$ be a
  globally generic cuspidal automorphic representation of $\GSp_4$
  over $F^+$ of weight $\mu$ for some $\mu$. Assume that for some finite place $v$ the local
  component $\pi_v$ is an unramified twist of the Steinberg
  representation. Then there is a Galois representation associated to
  $\pi$.
\end{thm}

It is now straightforward to use the results of the previous sections
to deduce a theorem about companion forms for automorphic
representations of $\GSp_4$ over $F^+$.

\begin{thm}\label{Companion forms for GSp_4} Let $F^+$ be a totally real field. Let $l\ge 5$ be a prime
  number such that $[F^+(\zeta_l):F^+]>2$. Fix $\iota:\Qlbar\isoto\C$. Suppose that \[\rhobar:G_{F^+}\to\GSp_4(\Flbar)\] is an
  irreducible representation, and let $n$ be an integer such that
  $\overline{\epsilon}^n$ is an unramified character of $G_{F^+}$. Suppose that $\rhobar$ satisfies the
following assumptions.
  \begin{enumerate}
   \item There are finite fields $\F_l \subset k\subset k'$ such that
     $\Sp_4(k)\subset\rhobar(G_{F^+})\subset (k')^\times\GSp_4(k)$.
  \item The representation $\rhobar$ has a lift which is $\GSp_4$-ordinarily
    automorphic and generic of level prime to $l$, with similitude factor $\psi$, say.
  \item Define $\psi_n:=\psi\epsilon^n\tilde{\omega}^{-n}$, where
    $\tilde{\omega}$ is the Teichm\"{u}ller lift of the mod $l$
    cyclotomic character (so $\overline{\psi}_n=\overline{\psi}$, and
    $\psi_n$ is crystalline). There is a
    $\lambda\in(\Z^4_+)^{\Hom(F^+,\Qlbar)}$ such that
    \begin{itemize}
    \item for every place $v|l$ of $F^+$, $\rhobar|_{G_{F^+_v}}$ has an 
      ordinary crystalline symplectic lift of weight
      $(\lambda_\tau)_{\tau}$ (where the
      indexing set runs over the embeddings $\tau\in\Hom(F^+,\Qlbar)$
      inducing $v$)
      and similitude factor $\psi_n$.
    \end{itemize}
  \end{enumerate}
  Then $\rhobar$ has an ordinary crystalline symplectic lift $\rho$ of weight
  $\lambda$ and similitude factor $\psi_n$, which is
  $\GSp_4$-ordinarily automorphic of level prime to $l$ and generic. If $F^+=\Q$
  then $\rho$ is also $\GSp_4$-ordinarily automorphic of level prime to $l$ and holomorphic.

 Given any set of places $S$ of $F^+$, and an inertial type $\tau_v$ for each $v\in S$ not dividing $l$ such that
$\rhobar|_{G_{F^+_v}}$ has a symplectic lift of type $\tau_v$ and
similitude factor $\psi_n$, $\rho$ can be
chosen to have type $\tau_v$ at $v$ for all places $v\in S$,
$v\nmid l$. More precisely, given a choice of a component of each ring
$R^{sympl,\tau_v,\psi_n}$ ($v\in S$, $v\nmid
l$) and  $R_{\rhobar|_{G_{F^+_v}}}^{sympl,\triangle_\lambda,cr,\psi_n}$
($v|l$), $\rho$ may be chosen so as to
give a point on each of these components.
\end{thm}
\begin{proof}
  This follows from Theorems \ref{transfer from gsp4 to gl4} and
  \ref{companion forms symplectic and gl4}. Note that if $\pi$ is a
  globally generic automorphic representation of $\GSp_4$ with
  $\rhobar_{\pi,\iota}\cong\rhobar$, then the transfer of $\pi$ to
  $\GL_4$ is cuspidal (because $\rhobar$ is irreducible). Conversely,
  if $\Pi$ is a RAESDC automorphic representation of $\GL_4(F^+)$ with
  $\Pi^\vee\cong\chi\Pi$, and $\rho_{l,\iota}(\Pi)$ is symplectic, it
  follows that $L^S(s,\Pi,\bigwedge^2\otimes\chi^{-1})$ has a pole at
  $s=1$ (because the corresponding statement is true for
  $\rho_{l,\iota}(\Pi)$).

In the case $F^+=\Q$, the fact that $\rho$ is also
$\GSp_4$-automorphic and holomorphic follows from Proposition 1.5 of \cite{MR2234860}
(because our assumptions on $\rhobar$ obviously imply that if
$\rhobar\cong\rhobar_{\pi,\iota}$ then $\pi$ is neither CAP nor weak endoscopic).
\end{proof}
In many cases we can make this rather more explicit, just as in the
unitary case.

\begin{lem}
  \label{lem:Ordinary always lifts in regular weight - symplectic
    version} Let $M$ be a finite extension of $\Ql$. Take
  $\lambda\in(\Z^4_+)^{\Hom(M,\Qlbar)}$. Let $E$ be a finite extension
  of $\Ql$ with residue field $k$. Let $\psi_i$, $1\le i\le 4$, be
  crystalline characters $G_M\to E^\times$, with
  $\psi_i|_{I_M}=\chi_i^\lambda|_{I_M}$ in the notation of Definition \ref{defn: chars associated to lambda}. Assume that
  $\psi_1\psi_4=\psi_2\psi_3$. Suppose that $\rhobar:G_M\to\GSp_4(k)$
  is of the form \[ \left(\begin{matrix} \overline{\mu}_{1} & * & * &*
      \cr 0 & \overline{\mu}_{2} & * &* \cr 0 & 0 & \overline{\mu}_3
      &*\cr 0 & 0 & 0& \overline{\mu}_4
    \end{matrix} \right)
  \]where $\overline{\psi}_i=\overline{\mu}_i$ for $1\le i\le
  4$. Suppose that none of the characters
  $\overline{\mu}_{i}\overline{\mu}_{j}^{-1}$, $i<j$, are equal to $\overline{\epsilon}$. Then
  $\rhobar$ has a lift to a crystalline representation
  $\rho:G_M\to\GSp_4(E)$ of the form \[ \left(\begin{matrix} \psi_{1}
      & * & * &* \cr 0 & \psi_{2} & * &* \cr 0 & 0 & \psi_3
      &*\cr 0 & 0 & 0& \psi_4
    \end{matrix} \right)
  \]
\end{lem}
\begin{proof}This is proved in exactly the same way as Lemma
  \ref{lem:Ordinary always lifts in regular weight}.
\end{proof}

Just as in section \ref{Refined
  liftings for l=p}, we can consider ordinary crystalline lifts of a
  particular form. Given $\rhobar$, $\lambda$ as in the previous lemma
  (but no longer requiring that the characters
  $\overline{\mu}_i\overline{\mu}_j^{-1}\ne \epsilon$), we can
  consider ordinary lifts where we demand that
  $\psi_i|_{I_M}=\chi_i^\lambda|_{I_M}$ and 
  $\overline{\psi}_i=\overline{\mu}_i$, $1\le i\le 4$.This gives a
  deformation ring
  $R_{\rhobar,\overline{\mu}}^{sympl,\triangle_\lambda,cr,\psi}$, and the
  following lemma may be proved in exactly the same way as Lemma \ref{lem: picking components of the ordinary crystalline
    deformation ring}.

\begin{lem}\label{lem: picking components of the ordinary crystalline
    deformation ring, symplectic case}
  After inverting $l$, the morphism  $\Spec
  R_{\rhobar,\overline{\mu}}^{sympl,\triangle_\lambda,cr,\psi}\to
  \Spec R_{\rhobar}^{sympl,\triangle_\lambda,cr,\psi}$ becomes a closed
  immersion identifying $\Spec
  R_{\rhobar,\overline{\mu}}^{sympl,\triangle_\lambda,cr,\psi}[1/l]$ with a union of irreducible
  components of $ \Spec R_{\rhobar}^{sympl,\triangle_\lambda,cr,\psi}[1/l]$.
\end{lem}

\begin{thm}\label{Companion forms for GSp_4 explicit} Let $F^+$ be a totally real field. Let $l\ge 5$ be a prime
  number such that $[F^+(\zeta_l):F^+]>2$. Fix $\iota:\Qlbar\isoto\C$. Suppose that \[\rhobar:G_{F^+}\to\GSp_4(\Flbar)\] is an
  irreducible representation. Suppose that the following conditions hold.
  \begin{enumerate}
   \item There are finite fields $\F_l \subset k\subset k'$ such that
     $\Sp_4(k)\subset\rhobar(G_{F^+})\subset (k')^\times\GSp_4(k)$.
  \item The representation $\rhobar$ has a lift which is
    $\GSp_4$-ordinarily automorphic and generic of level prime to $l$.
  \item There is a $\lambda\in(\Z^4_+)^{\Hom(F^+,\Qlbar)}$ such that
    \begin{itemize}
    \item $m:=\lambda_{\tau,1}+\lambda_{\tau,4}=\lambda_{\tau,2}+\lambda_{\tau,3}$ is independent of
      $\tau$,
and
    \item for every place $v|l$, $\rhobar|_{G_{F^+_v}}$ is isomorphic
      to a representation \[
  \left(\begin{matrix} \overline{\mu}_{v,1} & * & * &* \cr 0 &
      \overline{\mu}_{v,2} &  * &* \cr
0 & 0  &
      \overline{\mu}_{v,3} &*\cr 0 & 0 & 0&
     \overline{\mu}_{v,4}
    \end{matrix} \right)
  \]where none of $\overline{\mu}_{v,i}\overline{\mu}_{v,j}^{-1}$,
  $i<j$, are equal to
  $\overline{\epsilon}$. Furthermore,
  $\overline{\mu}_{v,i}|_{I_{F^+_v}}=\overline{\chi}_i^{\lambda_v}|_{I_{F^+_v}}$
  for each $i$ (in the notation of Definition \ref{defn: chars associated to lambda}).
    \end{itemize}
  \end{enumerate}
  Then $\rhobar$ has an ordinary crystalline symplectic lift $\rho$ of weight
  $\lambda$, which is
  $\GSp_4$-ordinarily automorphic of level prime to $l$ and generic, with
  similitude factor $\psi$, say. Furthermore
  $\psi\epsilon^{m+3}$ is a finite
  order character, and for every place $v|l$, $\rho|_{G_{F^+_v}}$
  is isomorphic to a representation of the form \[ \left(\begin{matrix} \psi_{v,1}
      & * & * &* \cr 0 & \psi_{v,2} & * &* \cr 0 & 0 & \psi_{v,3}
      &*\cr 0 & 0 & 0& \psi_{v,4}
    \end{matrix} \right)
  \]where the $\psi_{v,i}$ are crystalline characters such
  that 
  $\overline{\psi}_{v,i}=\overline{\mu}_{v,i}$ and
  $\psi_{v,i}|_{I_{F^+_v}}=\chi_{i}^{\lambda_v}|_{I_{F^+_v}}$. Finally, if $F^+=\Q$
  then $\rho$ is also $\GSp_4$-ordinarily automorphic  and
  holomorphic (of level prime to l).
\end{thm}
\begin{proof}
  This follows from Theorem \ref{Companion forms for GSp_4}, together with Lemma \ref{lem:Ordinary always lifts in regular weight - symplectic
    version} and Lemma \ref{lem: picking components of the ordinary crystalline
    deformation ring, symplectic case}.
\end{proof}

\begin{rem}
  It is expected that whenever $\pi$ is a cuspidal automorphic
  representation of $\GSp_4(\A_M)$, $M$ a number field, and $\pi$ is
  neither CAP nor weak endoscopic, then $\pi$ is stable. In the
  special case that $\pi$ is a discrete series representation at each
  infinite place, this means that if $\pi=\pi_f\otimes\pi_{\infty}$
  (with $\pi_f$, $\pi_\infty$ respectively denoting the finite and
  infinite factors of $\pi$) then $\pi_f\otimes\pi_\infty'$ is also
  automorphic for any $\pi'_\infty$ in the same $L$-packet as
  $\pi_\infty$, i.e. we are free to change between holomophic and
  generic discrete series at any infinite place. Assuming this result,
  which is expected to follow from Arthur's work on the trace formula
  (cf. \cite{MR2058604}), one could conclude that the representation
  $\rho$ in the above theorems is also $\GSp_4$-automorphic and
  holomorphic, even if $F^+\neq \Q$.
\end{rem}

\bibliographystyle{amsalpha} \bibliography{geegeraghty} 
\end{document}